\documentclass[11pt, a4paper]{amsart}
\usepackage{verbatim}
\usepackage[dvips]{color}
\usepackage{amssymb}
\usepackage[active]{srcltx}
\usepackage{latexsym}
\usepackage[dvipsnames, svgnames]{xcolor}
\usepackage{amsmath}
\usepackage{breakcites}
\usepackage[colorlinks,citecolor=.,hypertexnames=false,urlcolor=urlc,breaklinks=true, pagebackref, breaklinks=true]{hyperref}
 \usepackage{float}
\usepackage{mathrsfs}
\usepackage{cite}
\usepackage{esint}
\usepackage{comment}

\allowdisplaybreaks

\newtheorem{theorem}{Theorem}[section]

\newtheorem{lemma}{Lemma}[section]
\newtheorem{corollary}{Corollary}[section]

\theoremstyle{definition}
\newtheorem{remark}{Remark}[section]
\newtheorem{definition}{Definition}

\floatstyle{ruled} \restylefloat{figure} \restylefloat{table}

\numberwithin{equation}{section}

\newcommand{\re}{ \mathbf{R}}%
\newcommand{\beq}{\begin{equation}}
\newcommand{\bea}[1]{\begin{array}{#1} }
\newcommand{\eeq}{ \end{equation}}
\newcommand{\ea}{ \end{array}}

\newcommand{\al}{\alpha}
\newcommand{\ga}{\gamma}

\newcommand{\de}{\delta}
\newcommand{\ds}{\displaystyle}
\newcommand{\rar}{\mbox{$\rightarrow$}}

\newcommand{\ran}{\rangle}
\newcommand{\lan}{\langle}

\newcommand{\ar}{\partial}
\newcommand{\si}{\sigma}

\newcommand{\Om}{\Omega}

\newcommand{\hs}[1]{\mbox{$ \hspace{#1}$}}

\newcommand{\sem}{\setminus}

\newcommand{\De}{\Delta}

\def \O {{\Omega}}

\def \d {{\delta}}

\newcommand{\dist}{\operatorname{dist}}

\newcommand{\ree}{\mathbb{R}^{n+1}}

\newcommand{\W}{\mathcal{W}}

\def\H{\mathcal H}

\renewcommand{\d}{\, \mathrm{d}}

\newcommand{\pom}{\partial\Omega}

\def\Xint#1{\mathchoice
{\XXint\displaystyle\textstyle{#1}}%
{\XXint\textstyle\scriptstyle{#1}}%
{\XXint\scriptstyle\scriptscriptstyle{#1}}%
{\XXint\scriptscriptstyle%
\scriptscriptstyle{#1}}%
\!\int}
\def\XXint#1#2#3{{\setbox0=\hbox{$#1{#2#3}{%
\int}$ }
\vcenter{\hbox{$#2#3$ }}\kern-.6\wd0}}
\def\barint{\,\Xint -} 
\def\bariint{\barint_{} \kern-.4em \barint}
\def\bariiint{\bariint_{} \kern-.4em \barint}
\renewcommand{\iint}{\int_{}\kern-.34em \int} 
\renewcommand{\iiint}{\iint_{}\kern-.34em \int} 

\def\mean#1{\mathchoice%
          {\mathop{\kern 0.2em\vrule width 0.6em height 0.69678ex depth -0.58065ex
                  \kern -0.8em \intop}\nolimits_{\kern -0.4em#1}}%
          {\mathop{\kern 0.1em\vrule width 0.5em height 0.69678ex depth -0.60387ex
                  \kern -0.6em \intop}\nolimits_{#1}}%
          {\mathop{\kern 0.1em\vrule width 0.5em height 0.69678ex
              depth -0.60387ex
                  \kern -0.6em \intop}\nolimits_{#1}}%
          {\mathop{\kern 0.1em\vrule width 0.5em height 0.69678ex depth -0.60387ex
                  \kern -0.6em \intop}\nolimits_{#1}}}

\def\vintslides_#1{\mathchoice%
          {\mathop{\kern 0.1em\vrule width 0.5em height 0.697ex depth -0.581ex
                  \kern -0.6em \intop}\nolimits_{\kern -0.4em#1}}%
          {\mathop{\kern 0.1em\vrule width 0.3em height 0.697ex depth -0.604ex
                  \kern -0.4em \intop}\nolimits_{#1}}%
          {\mathop{\kern 0.1em\vrule width 0.3em height 0.697ex depth -0.604ex
                  \kern -0.4em \intop}\nolimits_{#1}}%
          {\mathop{\kern 0.1em\vrule width 0.3em height 0.697ex depth -0.604ex
                  \kern -0.4em \intop}\nolimits_{#1}}}

\newcommand{\aveint}[2]{\mathchoice%
          {\mathop{\kern 0.2em\vrule width 0.6em height 0.69678ex depth -0.58065ex
                  \kern -0.8em \intop}\nolimits_{\kern -0.45em#1}^{#2}}%
          {\mathop{\kern 0.1em\vrule width 0.5em height 0.69678ex depth -0.60387ex
                  \kern -0.6em \intop}\nolimits_{#1}^{#2}}%
          {\mathop{\kern 0.1em\vrule width 0.5em height 0.69678ex depth -0.60387ex
                  \kern -0.6em \intop}\nolimits_{#1}^{#2}}%
          {\mathop{\kern 0.1em\vrule width 0.5em height 0.69678ex depth -0.60387ex
                  \kern -0.6em \intop}\nolimits_{#1}^{#2}}}
\DeclareMathOperator{\interior}{int}
\DeclareMathOperator{\supp}{supp}

\def\eqn#1$$#2$${\begin{equation}\label#1#2\end{equation}}
\def\charfn_#1{{\raise1.2pt\hbox{$\chi
_{\kern-1pt\lower3pt\hbox{{$\scriptstyle#1$}}}$}}}
\def\diam{\operatorname{diam}}
\def\qq1{q_*}
\def\q2{q_{**}}
\def\dist{\operatorname{dist}}

\def\er{\mathbb R}

\def\osc{\operatorname{osc}}

\def\O{\rm O}

\newdimen\vintbar
\def\vint{-\kern-\vintbar\int}

\def\H{\mathcal H}

\def\O{\mathcal O}

\def\W{\mathcal W}

\def\0{\boldsymbol 0}

\newcommand{\eps}{\epsilon}

\newtoks\by
\newtoks\paper
\newtoks\book
\newtoks\jour
\newtoks\yr
\newtoks\pages
\newtoks\vol
\newtoks\publ

\def\name[#1, #2]{#1 #2}
\def\ota{{\hbox{\bf ???}}}
\def\cLear{\by=\ota\paper=\ota\book=\ota\jour=\ota\yr=\ota
\pages=\ota\vol=\ota\publ=\ota}
\def\endpaper{\the\by, \textit{\the\paper},
{\the\jour} \textbf{\the\vol} (\the\yr), \the\pages.\cLear}
\def\endbook{\the\by, \textit{\the\book},
\the\publ, \the\yr.\cLear}
\def\endpap{\the\by, \textit{\the\paper}, \the\jour.\cLear}
\def\endproc{\the\by, \textit{\the\paper}, \the\book, \the\publ,
\the\yr, \the\pages.\cLear}

\renewcommand{\d}{\, \mathrm{d}}

\allowdisplaybreaks

\title[Square function estimates for the evolutionary p-Laplace equation]{Square function estimates for the\\ evolutionary p-Laplace equation}

\author{Kaj Nystr\"om}
\address{Department of Mathematics, Uppsala University, S-751 06 Uppsala, Sweden}
\email{kaj.nystrom@math.uu.se}

\date{\today}

\begin{document}
\begin{abstract} We prove novel (local) square function/Carleson measure estimates for non-negative solutions to the evolutionary $p$-Laplace equation in the complement of parabolic Ahlfors-David regular sets. In the case of the heat equation, the Laplace equation as well as the $p$-Laplace equation, the corresponding square function estimates have proven fundamental in symmetry and inverse/free boundary type problems, and in particular in the study of (parabolic) uniform rectifiability. Though the implications of the square function estimates are less clear for the evolutionary $p$-Laplace equation, mainly due its lack of homogeneity, we give some initial applications to parabolic uniform rectifiability, boundary behaviour and Fatou type theorems for $\nabla_Xu$.
\end{abstract}

\maketitle


\section{Introduction}
Given $p$, $1< p<\infty$, fixed, the evolutionary $p$-Laplace equation, often referred to as the $p$-parabolic equation, is the equation
\begin{equation} \label{basic eq}
 \partial_tu-\nabla_X\cdot (|\nabla_X u|^{p-2}\nabla_X u) = 0,
\end{equation}
where  $u=u(X,t)$, $(X,t)\in\mathbb R^n\times \mathbb R$, $n\geq 1$. This equation is degenerate when $p > 2$, and singular
when $1 < p < 2$, as the modulus of ellipticity $|\nabla_X u|^{p-2}$  tends to $0$ and $+\infty$, respectively, as $|\nabla_X u|\to 0$.  When $p=2$, the equation is linear and coincides with the heat equation. It is well-known that if $p\neq 2$, then solutions to \eqref{basic eq} display different behaviours depending on $p$ (degenerate or singular) and in this paper we will only be concerned with the degenerate case $p>2$. Note that the evolutionary $p$-Laplace equation is invariant under  standard Euclidean translation in $X$ and $t$, and under the scalings $(X,t)\to (rX,r^pt)$. However, the equation is not homogeneous: if $u$ is a solution, then in general $cu$ is not a solution unless $c=1$. Also, in contrast to the case $p=2$, for $p>2$ any initial perturbation is propagated with finite speed by the equation and, as a consequence, any form of the strong maximum principle for non-negative solutions fails.

The results established in this paper concern inverse/free boundary type problems, and problems concerning the boundary behaviour of non-negative solutions,  for the degenerate evolutionary $p$-Laplace equation in time-dependent domains. Our contribution is inspired by recent progress concerning
symmetry problems, inverse/free boundary type problems involving the heat equation and parabolic measure, and parabolic uniform rectifiability, see \cite{LeNy,BHMN,BHMN1}, and
by studies of the corresponding  problems, and the boundary behaviour of non-negative solutions, for the $p$-Laplace equation, $1<p<\infty$, see \cite{LN,LN1,LN2,LN3,LN4,LN5,LN6,LN7}. The essence is that in all of these papers there are crucial square function/Carleson measure estimates lurking in the background, estimates based on which the oscillation of the (spatial) gradient of the solution can be controlled near the boundary in a $\mathrm{L}^2$-sense. The purpose of this paper is to prove that similar estimates remain valid for the degenerate evolutionary $p$-Laplace equation. Though the implications of these square function estimates are less clear for the  evolutionary $p$-Laplace equation, mainly due to the lack of homogeneity of the equation, we will give some initial applications to parabolic uniform rectifiability, boundary behaviour and Fatou type theorems for $\nabla_Xu$.

To be more precise we need to introduce some more notation, but we refer to the bulk of the paper for definitions. Let $\Sigma$ be a closed subset of $\mathbb R^{n+1}$ which is parabolic Ahlfors-David regular. Given $X\in\mathbb R^n$, let $B(X,r)$ denote the open ball in $\mathbb R^n$, centered at $X$ and of radius $r$. Let
\begin{align}\label{cyl} C( X, t,r ) \, := \, \{ ( Y, s )\in \mathbb R^{ n + 1 } :  Y\in B(X,r),\  | t - s | < r^2  \},
\end{align}
 whenever $(X,t)\in
\mathbb R^{n+1}$, $r>0$. We call $C(X,t,r) $ a parabolic cylinder of size $r$. Let $\delta(Y,s):=\dist(Y,s,\Sigma)$ denote the parabolic distance from $(Y,s)\in\mathbb R^{n+1}\setminus\Sigma$ to $\Sigma$. We let $\diam (\Sigma)$ denote the parabolic diameter of $\Sigma$.

The following square function estimate is a consequence of the \cite{LeNy} and Lemma \ref{partialADR} stated below.

\begin{theorem}\label{thm1-} Let $\Sigma$ be a closed subset of $\mathbb R^{n+1}$ which is parabolic Ahlfors-David regular with constant $M$, let  $\Omega:= \mathbb R^{n+1}\setminus \Sigma$. Let
$(X_0,t_0)\in \Sigma$, $r_0\in (0,\diam (\Sigma)/2)$. Assume  that $u$ is a non-negative function in $\Omega\cap C(X_0,t_0,2r_0)$  which satisfies $(\partial_t-\Delta_X) u=0$ in  $\Omega\cap C(X_0,t_0,2r_0)$. Assume in addition that there is a constant $\gamma$,  $1\leq\gamma<\infty$  such that
\begin{align}\label{boundshh} \delta(Y,s)|\nabla_X^2u(Y,s)|+|\nabla_X u(Y,s)| +(\delta(Y,s))^{-1}u(Y,s)\leq \gamma,\end{align}
for all $(Y,s)\in \Omega\cap C(X_0,t_0,r_0)$.  Then there exists a constant $c=c(n,M)\in (1,\infty)$ such that
\begin{align*}
\iiint_{\Omega\cap C(X,t,r)}
\bigl (|\nabla_X^2 u(Y,s)|^2+|\partial_t  u(Y,s)|^2\bigr )\, u(Y,s) \, \d Y\d s \lesssim r^{n+1},\end{align*}
whenever $(X,t)\in \Sigma$, $r>0$, $C(X,t,r)\subset C(X_0,t_0,r_0/c)$, and where the implicit constant depends only  on $n$, $M$ and the constant $\gamma$ in \eqref{boundshh}.
\end{theorem}

Theorem \ref{thm1-} states that
\begin{equation}\label{cara}
\bigl (|\nabla_X^2 u(Y,s)|^2+|\partial_t  u(Y,s)|^2\bigr )\, u(Y,s) \, \d Y\d s
\end{equation}
is a Carleson measure on $\Omega\cap C(X_0,t_0,r_0/c)$.  Note that  \eqref{boundshh} can, by interior regularity estimates for the heat equation, be replaced by the sufficient condition
\begin{align}\label{boundsa} u(Y,s)\leq \gamma \delta(Y,s)\mbox{ for all $(Y,s)\in \Omega\cap C(X_0,t_0,r_0)$.}\end{align}
Obviously \eqref{boundshh} and \eqref{boundsa} implies that $u(Y,s)\to 0$ in $\Omega\cap C(X_0,t_0,r_0)$ as $\delta(Y,s)\to 0$.

 In \cite{LeNy}, see also \cite{N2006,N2012}, the square function estimate in Theorem \ref{thm1-} was used in Lip(1,1/2) domains to prove certain symmetry theorems for the Green function associated to the heat equation. A crucial step in \cite{LeNy} was to prove that the imposed over-determined boundary condition implies that the boundary, originally assumed to be only Lip(1,1/2), is in fact regular Lip(1,1/2), i.e., parabolic uniform rectifiable. The stated square function estimate is fundamental to that argument.  Recently, in \cite{BHMN,BHMN1} this part of \cite{LeNy} was revisited in the context of caloric measure. Indeed, assuming appropriate background hypotheses and that the caloric measure has the weak $A_\infty$ property with respect to the surface measure on $\Sigma$, in \cite{BHMN1}, see also \cite{BHMN}, it is proved  that $\Sigma$ is parabolic uniform rectifiable. A key initial step in \cite{BHMN1} is to prove that the stated hypotheses implies that  $\Sigma\subset \ree$
satisfies the parabolic {weak  half-space approximation} property. The proof of this, as well as the argument in  \cite{BHMN}, rely on a
version of Theorem \ref{thm1-} but in certain dyadic sawtooth/Whitney regions.

In applications, see for example, \cite{LeNy,N2006,N2012,N2006,BHMN,BHMN1}, the Carleson measure estimate in \eqref{cara} is connected to geometry through the fact that we usually do not only have \eqref{boundsa}, but also the lower bound $\gamma^{-1}\delta(Y,s)\leq u(Y,s)$, either in all of $\Omega\cap C(X_0,t_0,r_0)$ or in certain dyadic sawtooth/Whitney regions. Based on this, \eqref{cara} implies that also \begin{equation}\label{caraml}
\bigl (|\nabla_X^2 u(Y,s)|^2+|\partial_t  u(Y,s)|^2\bigr )\, \delta(Y,s) \, \d Y\d s
\end{equation}
is a Carleson measure in the appropriate setting, and this is the result that is actually used to conclude information about $\Sigma$.

Focusing briefly on elliptic problems, the importance of this type of square function estimates for the Laplace equation, as well as the $p$-Laplace equation, was established in
\cite{LV0,LV1,LV2}. In this case $E$ is a closed subset of $\mathbb R^{n}$ which is Ahlfors-David regular of dimension $(n-1)$ with constant $M$, $D:= \mathbb R^{n}\setminus E$. In the case of the  Laplace equation,  it is proved that if $X_0\in E$,  if $u$ is a non-negative function in $D\cap B(X_0,2r_0)$, $\Delta_X u=0$ in  $D\cap B(X_0,2r_0)$, and if $u(Y)\leq \gamma \delta(Y)$ for all $Y\in D\cap B(X_0,r_0)$, then
\begin{equation}\label{cara+}
 |\nabla_X^2 u(Y)|^2\, u(Y) \, \d Y
\end{equation}
is a Carleson measure on $D\cap B(X_0,r_0/c)$. In the case of the $p$-Laplace equation, $1<p<\infty$, $u$ is instead assumed to be a solution to the $p$-Laplace equation
$\nabla_X\cdot(|\nabla_Xu|^{p-2}\nabla_Xu)=0$ in $D\cap B(X_0,2r_0)$. In this case the conclusion is, under additional assumptions concerning the non-degeneracy of $|\nabla_Xu|$, that
\begin{equation}\label{cara++}
|\nabla_X u(Y)|^{p}\, |\nabla_X^2 u(Y)|^2\, u(Y) \, \d Y
\end{equation}
is a Carleson measure on $D\cap B(X_0,r_0/c)$, see Theorem 1 in \cite{Le} and \cite{LV2}. In \eqref{cara+} and \eqref{cara++} the constants in the Carleson measure estimates depend at most on $n,M,\gamma$ and $p$. We refer to \cite{Le} for an excellent and lucid survey of these developments in the context of symmetry, inverse/free boundary type problems, and uniform rectifiability. Building on
\cite{LV0,LV1,LV2}, in \cite{HLMN} it is proved that if $E\subset \mathbb R^{n}$, $n\ge 3$,  is a Ahlfors-David regular set of dimension $(n-1)$, then the weak-$A_\infty$ property of harmonic measure, for the open set
$D= \mathbb R^n\setminus E$, implies uniform rectifiability of $E$. More generally, in \cite{HLMN} a similar result is established for the Riesz measure, $p$-harmonic measure,
associated to the $p$-Laplace operator, $1<p<\infty$. In \cite{HLMN}, versions of the Carleson measure estimates in \eqref{cara+} and \eqref{cara++}, in certain dyadic sawtooth/Whitney regions,  are crucial to the arguments.

The purpose of this paper is to establish a version of Theorem \ref{thm1-} but for the degenerate evolutionary $p$-Laplace equation. In particular, we prove the following theorem.

\begin{theorem}\label{thm1-a} Let $p$, $2<p<\infty$, be fixed. Let $\Sigma$ be a closed subset of $\mathbb R^{n+1}$ which is parabolic Ahlfors-David regular with constant $M$, let  $\Omega:= \mathbb R^{n+1}\setminus \Sigma$. Let
$(X_0,t_0)\in \Sigma$, $r_0\in (0,\diam (\Sigma)/2)$. Assume  that $u$ is a  non-negative function in $\Omega\cap C(X_0,t_0,2r_0)$  which satisfies $\partial_tu-\nabla_X\cdot(|\nabla_Xu|^{p-2}\nabla_Xu)=0$ in  $\Omega\cap C(X_0,t_0,2r_0)$. Assume in addition that
\begin{align}\label{boundsaapall}
 \mbox{$|\nabla_Xu(Y,s)|>0$ for all $(Y,s)\in\Omega\cap C(X_0,t_0,2r_0)$},
 \end{align}
 and that there is a constant $\gamma$, $1\leq \gamma<\infty$,  such that
\begin{align}\label{boundsaapa} \delta(Y,s)|\nabla_X^2u(Y,s)|+|\nabla_X u(Y,s)| +(\delta(Y,s))^{-1}u(Y,s)\leq \gamma,\end{align}
for all $(Y,s)\in \Omega\cap C(X_0,t_0,r_0)$.  Then there exists a constant $c=c(n,M)\in (1,\infty)$ such that the following holds. Let $q\geq p$ and let $r:= q-2p+4$. Then
\begin{align*}
 \iiint_{\Omega\cap C(X,t,r)}\bigl (|\nabla_X u(Y,s)|^{q}\, |\nabla_X^2 u(Y,s)|^2+|\nabla_X u(Y,s)|^{r}\, |\partial_t u(Y,s)|^2\bigr )\, u(Y,s) \, \d Y\d s \lesssim r^{n-1},
\end{align*}
whenever $(X,t)\in \Sigma$, $r>0$, $C(X,t,r)\subset C(X_0,t_0,r_0/c)$, and with implicit constant depending only  on $n$, $p$, $M$, $\gamma$, and  $q$.
\end{theorem}

Theorem \ref{thm1-a} states, subject to the stated restrictions on $(q,r)$, that
\begin{align}\label{apa}
 \bigl (|\nabla_X u(Y,s)|^{q}\, |\nabla_X^2 u(Y,s)|^2+|\nabla_X u(Y,s)|^{r}\, |\partial_t u(Y,s)|^2\bigr )\, u(Y,s) \, \d Y\d s
\end{align}
is a Carleson measure on $\Omega\cap C(X_0,t_0,r_0/c)$. Note that $(q,r)=(p,4-p)$ are admissible exponents in Theorem \ref{thm1-a}, and hence \begin{align}\label{apakka}
 \bigl (|\nabla_X u(Y,s)|^{p}\, |\nabla_X^2 u(Y,s)|^2+|\nabla_X u(Y,s)|^{-(p-4)}\, |\partial_t u(Y,s)|^2\bigr )\, u(Y,s) \, \d Y\d s
\end{align}
is a Carleson measure on $\Omega\cap C(X_0,t_0,r_0/c)$. For stationary functions, \eqref{apakka} is consistent with the measure in \eqref{cara++}.

We also prove the following result which should be seen as a corollary to Theorem \ref{thm1-a}.

\begin{corollary}\label{thm1-b} Let $p$, $\Sigma$,  $\Omega$, $(X_0,t_0)$, $r_0$, and $u$ be as in the statement of Theorem \ref{thm1-a}. Assume \eqref{boundsaapall} and \eqref{boundsaapa}, and in addition that
\begin{align}\label{boundsaapacor} \delta^2(Y,s)|\nabla_X^3u(Y,s)|\leq \gamma,\end{align}
for all $(Y,s)\in \Omega\cap C(X_0,t_0,r_0)$. Then there exists a constant $c=c(n,M)\in (1,\infty)$ such that the following holds. For $a\geq 1$ and $b\geq 1$ sufficient large, depending on $p$, we have
\begin{align*}
\iiint_{\Omega\cap C(X,t,r)}\bigl (|\nabla_X u(Y,s)|^{a}|\nabla_X^3u(Y,s)|^2+|\nabla_X u(Y,s)|^{b}|\partial_t\nabla_Xu(Y,s)|^2\bigr )u^3(Y,s)\, \d Y\d s &\lesssim r^{n-1},
\end{align*}
whenever $(X,t)\in \Sigma$, $r>0$, $C(X,t,r)\subset C(X_0,t_0,r_0/c)$, and with implicit constant depending only  on $n$, $p$, $M$, $\gamma$, $a$ and $b$.
\end{corollary}

Concerning Theorem \ref{thm1-a} and Corollary \ref{thm1-b}, several comments are in order.

First, weak solutions to the equation in \eqref{basic eq} are in general only $C^{1,\alpha}$ regular, but if \eqref{boundsaapall} holds then $u$ is a smooth classical solution to the equation in \eqref{basic eq},  see subsection \ref{nondd}. Note also that to even have the quantities appearing in \eqref{apa}, as well as similar quantities appearing in the proof of Theorem \ref{thm1-a}, well defined and finite we also need \eqref{boundsaapall}. Differently, we could simply have assumed
that $u$ is a smooth solution in $\Omega\cap C(X_0,t_0,2r_0)$ and formulated a result as in  Corollary \ref{thm1-b} (saying that there exists sufficient large $a\geq 1$ and $b\geq 1$). Concerning the Schauder estimates for higher order spatial partial derivatives stated in \eqref{boundsaapa}, \eqref{boundsaapacor},  if $(\delta(Y,s))^{-1}u(Y,s)\leq \gamma$, then these estimates can be verified if, for example,
\begin{align}\label{strongND}
\gamma^{-1}\leq |\nabla_X u(Y,s)| \mbox{ for all }(Y,s)\in\Omega\cap C(X_0,t_0,2r_0),
\end{align}
we here again refer to subsection \ref{nondd}. If \eqref{strongND} holds, then \eqref{boundsaapa} and \eqref{strongND} imply that there is a constant $\tilde\gamma$, $1\leq\tilde\gamma<\infty$, such that
\begin{align}\label{strongND+}
\tilde\gamma^{-1}\leq |\nabla_X u(Y,s)|\leq \tilde\gamma,
\end{align}
for all $(Y,s)\in\Omega\cap C(X_0,t_0,r_0)$. In particular,  the statements of Theorem \ref{thm1-a} and Corollary \ref{thm1-b} implicitly  assume some strong non-degeneracy conditions on
$|\nabla_X u(Y,s)|$.

Second, while Theorem \ref{thm1-a} and Corollary \ref{thm1-b} may seem quite similar they are stated based on a distinction. Indeed,  consider  the linear operators
\begin{align}
\mathcal{H}&:=\partial_t-\nabla_X\cdot (A(X,t)\nabla_X),\notag\\
\tilde {\mathcal{H}}&:= \partial_t-\nabla_X\cdot (\tilde A(X,t)\nabla_X),
\end{align}
where
\begin{align}
A_{ij}(Y,s)&:=|\nabla u(Y,s)|^{p-2}\delta_{ij},\notag\\
 \tilde A_{ij}(Y,s)&:=|\nabla_Xu(Y,s)|^{p-4}[(p-2)u_{x_i}(Y,s)u_{x_j}(Y,s)+\delta_{ij}|\nabla_X u(Y,s)|^2].
\end{align}
Note that
\begin{align}
\min\{1,p-1\}|\nabla_Xu(Y,s)|^{p-2}|\xi|^2\leq \tilde A_{ij}(Y,s)\xi_i\xi_j\leq \max\{1,p-1\}|\nabla_Xu(Y,s)|^{p-2}|\xi|^2.
\end{align}
In particular, the constants of ellipticity of the linear operators $\mathcal{H}$ and  $\tilde {\mathcal{H}}$ are, at $(Y,s)$,  determined by $|\nabla_Xu(Y,s)|^{p-2}$. $\tilde {\mathcal{H}}$ is the operator obtained by formally differentiating the $p$-parabolic equation, i.e., if $u$ is a (weak) solution to \eqref{basic eq}, and if we let $v=u_{x_k}$, then formally $\tilde {\mathcal{H}}v=0$. Note also that if $u$ is a weak solution to \eqref{basic eq}, then
\begin{align}
{\mathcal{H}}u&=\partial_tu-\nabla_X\cdot (A(X,t)\nabla_Xu)=0,\notag\\
\tilde {\mathcal{H}}u&=\partial_tu-(p-1)\nabla_X\cdot (A(X,t)\nabla_Xu),
\end{align}
but in general $\tilde {\mathcal{H}}u$ is not equal to $0$.  In particular, due to the lack of homogeneity of the evolutionary $p$-Laplace equation, $u$ and its spatial partial derivatives solve different linear parabolic partial differential equations. Now the point is that Theorem \ref{thm1-a} is  a statement about solutions to the $p$-parabolic equation, while jointly Theorem \ref{thm1-a} and Corollary \ref{thm1-b} can be seen as statements concerning the coefficients of the linear operators ${\mathcal{H}}$ and $\tilde {\mathcal{H}}$. Indeed, if
\begin{align}\label{strongND+ml}
\tilde\gamma^{-1}\leq \min\{|\nabla_X u(Y,s)|,(\delta(Y,s))^{-1}u(Y,s)\}\leq \max\{|\nabla_X u(Y,s)|,(\delta(Y,s))^{-1}u(Y,s)\} \leq \tilde\gamma,
\end{align}
for all $(Y,s)\in\Omega\cap C(X_0,t_0,r_0)$, then Theorem \ref{thm1-a} and Corollary \ref{thm1-b} imply that
\begin{equation}\label{caraa}
\bigl (|\nabla_X A(Y,s)|^2+|\nabla_X \tilde A(Y,s)|^2\bigr )\, \delta (Y,s)\, \d Y\d s,
\end{equation}
as well as
\begin{equation}\label{carab}
\bigl (|\partial_t A(Y,s)|^2+|\partial_t\tilde A(Y,s)|^2\bigr )\, \delta^3(Y,s) \, \d Y\d s,
\end{equation}
are Carleson measures on $\Omega\cap C(X_0,t_0,r_0/c)$. This implies, as discussed in subsection \ref{pmeasure} below, that at least in the setting of regular Lip(1,1/2) domains, the parabolic measures associated to ${\mathcal{H}}$ and $\tilde {\mathcal{H}}$ satisfy scale invariant estimate with respect to surface measure in the sense of $A_\infty$.  The latter will be explored in our applications to parabolic uniform rectifiability, boundary behaviour and Fatou type theorems for $\nabla_Xu$.

Recall that the notion of parabolic uniformly rectifiable sets was introduced by the author, together with S. Hofmann and J. Lewis, in \cite{HLN}, \cite{HLN1}, and concerns  time-varying boundaries which lack differentiability, and which are locally not necessarily given by graphs. Instead {geometry is controlled by a local geometric square function, based on which key geometric information and structure can be extracted: this is captured in the notion of parabolic uniformly rectifiable sets.  The local geometric square function quantifies, on each scale, how the underlying set deviates from time-independent hyperplanes in a $L^2$-sense (mean square sense). Parabolic uniform rectifiability is the dynamic counterpart of the uniform rectifiability studied in the  monumental works of G. David and S. Semmes \cite{DS1}, \cite{DS2}. The notions of parabolic uniformly rectifiable sets and parabolic uniform rectifiability extract the geometrical theoretical essence of the (time-dependent) (regular) parabolic Lipschitz graphs introduced in \cite{H}, \cite{HL}, \cite{LM}, \cite{LS}. In these works the authors found the correct notion of (time-dependent) (regular) parabolic Lipschitz graphs from the point of view of parabolic singular integrals and parabolic/caloric measure. In particular, in the context of Lip(1,1/2) graphs, see below, a graph being regular Lip(1,1/2) is equivalent to the graph being parabolic uniform rectifiable.

As a first application of the square function/Carleson measure estimates in Theorem \ref{thm1-a} and Corollary \ref{thm1-b}, we consider parabolic uniform rectifiability and we prove the following theorem.

\begin{theorem}\label{Free}  Let $p$, $2<p<\infty$, be given. Let $\Omega\subset\mathbb R^{n+1}$ be a
(unbounded) Lip(1,1/2)   graph
domain with constant $b_1$ and with boundary $\Sigma$.  Let
$(X_0,t_0)\in \Sigma$, assume  that $u$ is a  non-negative function in $\Omega\cap C(X_0,t_0,2r_0)$  which is a weak solution to  $\partial_tu-\nabla_X\cdot(|\nabla_Xu|^{p-2}\nabla_Xu)=0$ in  $\Omega\cap C(X_0,t_0,2r_0)$.
 Assume that  there exist a constant $\gamma$,  $1\leq\gamma<\infty$,  such that
\begin{align}\label{boundsfree}\gamma^{-1}\leq  |\nabla_X u(Y,s)|,\ (\delta(Y,s))^{-1}u(Y,s)\leq \gamma,\end{align}
for all $(Y,s)\in \Omega\cap C(X_0,t_0,r_0)$. Then there exists a constant $c=c(n,b_1)\in (1,\infty)$ such that,
\begin{align}\label{conc1}
&\mbox{$\Sigma\cap C(X_0,t_0,r_0/c)$ is locally parabolic uniform rectifiable in the sense that this set}\notag\\
&\mbox{is a local regular Lip(1,1/2) graph, with constant $b_2=b_2(n,p,b_1,\gamma)$}.
\end{align}
\end{theorem}

As a second application, we establish a Fatou type theorem for $\nabla_X u$. The non-tangential cone $\Gamma (X,t)$,  the (time-independent) unit vector $\mathbf{n}(X,t)\in \mathbb R^n$, and the surface measure $\sigma$, are defined in the bulk of the paper.

\begin{theorem}\label{Free+} Let $p$, $2<p<\infty$, be given. Let $\Omega\subset\mathbb R^{n+1}$ be a
(unbounded) regular Lip(1,1/2)   graph
domain,
\begin{eqnarray*}
\Omega=\Omega_\psi=\{(x,x_n,t)\in\mathbb R^{n-1}\times\mathbb R\times\mathbb R:x_n>\psi(x,t)\},
\end{eqnarray*}
with constants $(b_1,b_2)$  and with boundary $\Sigma$.  Let
$(X_0,t_0)\in \Sigma$, assume  that $u$ is a non-negative function in $\Omega\cap C(X_0,t_0,2r_0)$  which is a weak solution to $\partial_tu-\nabla_X\cdot(|\nabla_Xu|^{p-2}\nabla_Xu)=0$ in  $\Omega\cap C(X_0,t_0,2r_0)$.
 Assume that  there exist a constant $\gamma$,  $1\leq\gamma<\infty$, such that \eqref{boundsfree} holds for all $(Y,s)\in \Omega\cap C(X_0,t_0,r_0)$. Then,
\begin{align}\label{conc2}
\nabla_X u ( X,t):=\lim_{\substack{(Y,s)\in \Gamma(X,t)\\ (Y,s)\to (X,t)}} \nabla_X u ( Y,s )
 \end{align}
 exists for $\sigma$-a.e. $(X,t)\in\Sigma\cap C(X_0,t_0,r_0)$. Furthermore, assume in addition that
 \begin{align}\label{boundsfreere}\gamma^{-1}\leq\langle\nabla_X u(Y,s),e_n\rangle,\end{align}
for all $(Y,s)\in \Omega\cap C(X_0,t_0,r_0)$. Then there exists,  if we let $u\equiv 0$ on $(\mathbb R^{n+1}\setminus\Omega)\cap C(X_0,t_0,r_0)$, a locally finite measure $\mu$ supported on $\Sigma\cap C(X_0,t_0,r_0)$ such that
\begin{align} \label{conc3}
\iiint |\nabla_Xu|^{p-2}\nabla_Xu\cdot\nabla_X\phi -u \partial_t\phi\, \d X \d t =-\iint \phi \d \mu
\end{align}
for all $ \phi \in C_0^\infty(C(X_0,t_0,r_0))$. Furthermore, for  $\sigma$-a.e. $(X,t)\in\Sigma\cap C(X_0,t_0,r_0)$
\begin{align} \label{conc4}
\nabla_Xu(X,t)=|\nabla_Xu(X,t)|\mathbf{n}(X,t)\mbox{ and }\d \mu(X,t)=|\nabla_Xu(X,t)|^{p-1}\d\sigma(X,t).
\end{align}
\end{theorem}

Concerning Theorem \ref{Free} and Theorem \ref{Free+}, it is naturally unclear in what applications \eqref{boundsfree} and \eqref{boundsfreere} can be verified. However, in the mathematical theory of free boundaries, conditions similar to one in \eqref{boundsfree} appear frequently as part of the analysis, see for example \cite{AW,ACS1,ACS2,CLW1,CLW2,CV,E,LeNy} for works devoted to the heat operator, and \cite{DH,TT} for some results concerning the evolutionary $p$-Laplace operator. \eqref{boundsfreere} implies that the level sets of the function $u$ close to the boundary are Lip(1,1/2)  graphs, at least locally in the given coordinate system.

\subsection{Organization of the paper} Section \ref{pre} is of preliminary type and we here introduce notation, weak solutions, parabolic Ahlfors-David regular sets and Whitney cubes.  In Section \ref{Proofmain} we prove Theorem \ref{thm1-a} and Corollary \ref{thm1-b}. In Section \ref{Upara} we introduce  parabolic uniform rectifiability, (regular) Lip(1,1/2) graphs and we state a Rademacher theorem for regular Lip(1,1/2) functions. In this section we also introduce parabolic measure for linear parabolic operators in divergence form, and we state the results we need concerning the $A_\infty$ property with respect to surface measure for parabolic measure. In Section \ref{Pest} we discuss estimates for the evolutionary $p$-Laplace equation. Section \ref{App} is devoted to the applications to  parabolic uniform rectifiability, boundary behaviour and Fatou type theorems for $\nabla_Xu$ mentioned above, and we here prove Theorem \ref{Free}  and Theorem \ref{Free+}. While we undoubtedly prove new results in this paper, it is also fair to say that paper also  has a bit of the character of a survey as we connect results and techniques from several fields of (parabolic) PDEs.


\section{Preliminaries}\label{pre}

\subsection{Notation}
 Points in  Euclidean space-time 
$ \mathbb R^{n+1} $ are denoted by $(X,t) = ( x_1,
 \dots,  x_n,t)$,  where $ X = ( x_1, \dots,
x_{n} ) \in \mathbb R^{n } $, $n\geq 1$, and $t$ represents the time-coordinate.  We let $ \d X $ denote  Lebesgue $ n $-measure on    $ \mathbb R^{n}$  and we let $ \d t $ denote  Lebesgue $1$-measure on    $ \mathbb R$. We let  $ \bar E, \ar E$,
  be the closure and boundary of the set $ E \subset
\mathbb R^{n+1}$. $  \lan \cdot ,  \cdot  \ran $  denotes  the standard inner
product on $ \mathbb R^{n} $ and we let  $  | X | = \lan X, X \ran^{1/2} $ be
the  Euclidean norm of $ X. $  We let $\|(X,t)\|:=|X|+|t|^{1/2}$ denote the parabolic length
of a space-time vector $(X,t)$. Given $(X,t), (Y,s)\in\mathbb R^{n+1}$ we let $$\dist(X,t,Y,s):=|X-Y|+|t-s|^{1/2},$$ and, more generally, we let
\begin{align}\label{dista}\dist(E_1,E_2) := \inf_{(X,t) \in E_1, (Y,s) \in E_2} \dist(X,t,Y,s),
\end{align}
denote the parabolic distance between  $E_1$ and $E_2$, where $E_1,E_2 \subseteq \mathbb R^{n+1}$.  $\dist( X,t, E ) $  is defined to equal the parabolic distance, defined with respect to $\dist(\cdot,\cdot)$, from
 $  (X,t) \in \mathbb R^{n+1} $ to $ E$. We let $\diam(E)$ denote the parabolic diameter of $E$, i.e. the diameter of $E$ as measured using the parabolic distance function. In addition, we let
\begin{align}\label{distb}H(E_1,E_2) := \max\bigl\{\sup_{(X,t) \in E_1} \dist(X,t,E_2),\sup_{(Y,s) \in E_2} \dist(Y,s,E_1)\bigr \},
\end{align}
 denote the parabolic Hausdorff distance between  $E_1$ and $E_2$, where $E_1,E_2 \subseteq \mathbb R^{n+1}$. Given $X\in\mathbb R^n$ we let $B(X,r)$ denote the open ball in $\mathbb R^n$, centered at $X$ and of radius $r$, and we recall that the parabolic cylinder $C(X,t,r) $ was introduced in
 \eqref{cyl}.

  \subsection{Functional setting} If $ U \subset \mathbb R^{n} $ is open and $ 1 \leq q \leq \infty, $ then by $ W^{1 ,q} ( U)$ we denote the space of equivalence classes of functions $ f $ with distributional gradient $ \nabla_Xf= ( f_{x_1}, \dots, f_{x_n} ), $ both of which are $ q $-th power integrable on $ U. $ Let \,
\[ \| f \|_{ W^{1,q} (U)} := \| f \|_{ L^q (U)} + \| \, | \nabla_X f | \, \|_{ L^q ( U)} \, \]
be the norm in $ W^{1,q} ( U ) $ where $ \| \cdot \|_{L^q ( U )} $ denotes the usual Lebesgue $q$-norm in $U$. $ C^\infty_0 (U )$ is the set of infinitely differentiable functions with compact support in $ U$ and we let $ W^{1 ,q}_0 ( U )$ denote the closure of $ C^\infty_0 (U )$ in the norm $\| \cdot\|_{ W^{1,q} (U)}$. $ W^{1,q}_{\rm loc} ( U ) $ is defined in the standard way. By $ \nabla_X \cdot $ we denote the divergence operator in $X$. Given $U\subset\mathbb R^n$, and $t_1<t_2$, we denote by $L^q(t_1,t_2,W^{1,q} ( U ))$ the space of functions such that for almost every $t$, $t_1\leq t\leq t_2$, the function $X\to u(X,t)$ belongs to $W^{1,q} ( U )$ and
\begin{equation*}
	\| u \|_{ L^q(t_1,t_2,W^{1,q} ( U ))}:=\bigl (\int\limits_{t_1}^{t_2}\iint\limits_U\bigl (|u(X,t)|^q+|\nabla_X u(X,t)|^q\bigr )\d X\d t\bigr )^{1/q} <\infty\,.
\end{equation*}
The spaces $L^q(t_1,t_2,W^{1,q}_0 ( U ))$ and $L^q_{\rm loc}(t_1,t_2,W^{1,q}_{\rm loc} ( U ))$ are defined analogously.

\subsection{Weak solutions} \label{sec:solution} Let $ \Omega \subset \mathbb R^{n+1} $ be an open set. Given $p$, $1<p<\infty$, we say that $ u $ is a weak solution to the equation in \eqref{basic eq} in $\Omega$, if the following is true whenever $U \subset \mathbb R^{n} $ is open,  $-\infty<t_1<t_2<\infty$,  and $\O:=U\times (t_1,t_2)\Subset\Omega$. First,  $u\in L^p_{\rm loc}(t_1,t_2,W^{1,p}_{\rm loc} ( U ))$, and second
\begin{equation}
	\label{1.1tr} \iiint |\nabla_Xu|^{p-2}\nabla_X u\cdot\nabla_X \phi -u\partial_t\phi\,  \d X\d t= 0,
\end{equation}
whenever $ \phi \in C_0^\infty(\O)$. If $ u $ is a weak solution to  \eqref{basic eq} in this sense, then we will  refer to $u$ as  a weak solution to the evolutionary $p$-Laplace equation, or the $p$-parabolic equation, and we will sometimes
refer to $u$ as being a $p$-parabolic function in $\Omega$. If \eqref{1.1tr} holds with $=$ replaced by $\geq$ ($\leq$) for all $ \phi \in C_0^\infty(\O)$, $\phi \geq 0$, then we will refer to $u$ as a weak supersolution (subsolution) to the $p$-parabolic equation. Note that for  $p \in (2,\infty)$ fixed,   by the regularity theory for weak solutions, see \cite{DB}, any $p$-parabolic function $u$ has a locally H{\"o}lder continuous (in space and time) representative.

 \subsection{{Parabolic} Hausdorff measure}
  Given $\eta \geq 0$, we let $\H^\eta$ denote
 standard $\eta$-dimensional Hausdorff measure.
  We also define a {parabolic} Hausdorff measure of {homogeneous}
  dimension $\eta$, denoted
  $\H_{\text p}^\eta$, in the same way that one defines standard Hausdorff measure, but instead using coverings
  by {parabolic} cubes. I.e., for $\delta>0$, and for $A\subset \mathbb R^{n+1}$, we set
  \[ \H_{\text{p},\delta}^\eta(A):= \inf \sum_k \diam(A_k)^\eta\,,
  \]
  where the infimum runs over all countable such coverings of $A$, $\{A_k\}_k$, with $\diam(A_k)\leq \delta$ for all $k$. We then define
  \[
  \H_{\text p}^\eta (A) := \lim_{\delta\to 0^+} \H_{\text{p},\delta}^\eta(A)\,.
  \]
As is the case for classical Hausdorff measure, $ \H_{\text{p}}^\eta$ is a Borel regular measure.
We refer the reader to \cite[Chapter 2]{EG} for a discussion of the basic properties of standard
Hausdorff measure.  The arguments in \cite{EG} adapt readily to treat $  \H_{\text{p}}^\eta$.
In particular, one obtains a measure equivalent to   $\H_{\text{p}}^\eta$ if one defines
$\H_{\text{p},\delta}^\eta$ in terms of coverings by arbitrary sets of parabolic diameter at most $\delta$, rather than
cubes.  As in the classical setting, we define the parabolic homogeneous dimension of a set
$A\subset \mathbb R^{n+1}$ by
\[\H_{\text{p}, \text{dim}}(A):= \inf\left\{ 0\leq \eta<\infty \,| \,\H_{\text p}^\eta(A)=0\right\}.
\]
We observe that $\H_{\text{p}, \text{dim}}(\ree)=n+2$.

\subsection{Surface measure}
Given a closed set $\Sigma \subset \mathbb R^{n+1}$
  of  homogeneous dimension $\H_{\text{p}, \text{dim}}(\Sigma)=n+1$, we
define  a surface measure on $\Sigma$ as the restriction of $\H_{\text{p}}^{n+1}$ to $\Sigma$, i.e.,
\begin{equation}\label{sigdef}
\sigma = \sigma_\Sigma:=  \H_{\text{p}}^{n+1}|_\Sigma\,.
\end{equation}
For  $ (X, t ) \in \Sigma $ and $r>0$, we let
$$\Delta(X,t,r):=\Sigma\cap C(X,t,r)\,.$$
The extremal time coordinates of $\Sigma$ will be denoted by $T_0=\inf\{t:\exists (X,t)\in\Sigma\}$ and $T_1=\sup\{t:\exists (X,t)\in\Sigma\}$. Throughout the rest of the paper we will,  for simplicity, consistently assume that
 \begin{align}\label{Assump}
 &\mbox{$\diam\Sigma=\infty$ and that $T_0=-\infty$ and $T_1=\infty$.}
 \end{align}

\subsection{Parabolic Ahlfors-David regular sets}

\begin{definition}\label{def1.ADR}{\bf (Parabolic Ahlfors-David Regularity).} Let
$\Sigma \subset \mathbb R^{n+1}$ be a closed set.  We say that $\Sigma$ is {parabolic Ahlfors-David regular}, parabolic ADR for short
(or simply p-ADR, or just ADR)
with constant $M\geq 1$,  if
\begin{equation} \label{eq1.ADRha}
M^{-1}\, r^{n+1} \leq \sigma(\Delta(X,t,r)) \leq M\, r^{n+1},\end{equation}
whenever $0<r<\infty$,  $(X,t)\in \Sigma$.
\end{definition}

\subsection{Whitney cubes} Let
$\Sigma \subset \mathbb R^{n+1}$ be a closed set and let $\Omega:=\mathbb R^{n+1}\setminus\Sigma$. We let $\mathcal{W}=\W(\Omega)$ denote a collection $\{I\}$ of (closed) dyadic parabolic Whitney cubes of $\Omega$, constructed so that the cubes in $\mathcal{W}$
form a covering of $\Omega$ with non-overlapping interiors, such that
\begin{equation}\label{eqWh1} 4\, {\rm{diam}}\,(I)\leq \dist(4 I,\Sigma) \leq  \dist(I,\Sigma) \leq 40 \, {\rm{diam}}\,(I),\end{equation}
and
\begin{equation}\label{eqWh2}\diam(I_1)\sim \diam(I_2), \mbox{ whenever $I_1$ and $I_2$ touch.}
\end{equation}
Given a small, positive parameter $\tau$, and given $I\in\W$, we let
\begin{equation}\label{eq2.3*}I^* =I^*(\tau) := (1+2\tau)I,
\end{equation}
denote corresponding fattened Whitney cubes.
We  fix $\tau$ sufficiently small so that the cubes $\{I^{*}\}$ retain the properties of the Whitney cubes and in particular that
\begin{equation}\label{eq2.3*g}\diam(I) \sim \diam(I^{*}) \sim \dist(I^{*},\Sigma) \sim \dist(I,\Sigma),
\end{equation}
where, strictly speaking, the implicit constants now also depend on $\tau$.

\subsection{Whitney regions and integration} Let $\Omega=\mathbb R^{n+1}\setminus\Sigma$ and consider the collection
of (closed) dyadic Whitney cubes $\mathcal{W}=\W(\Omega)$ introduced above, see \eqref{eqWh1} and \eqref{eqWh2}. Assuming that $\Sigma=\pom$ is parabolic ADR, $(X_0,t_0)\in \Sigma$, $r_0>0$, we, to this end, fix $c=c(n,M)\geq 1$ large enough to ensure that if $I\in \mathcal{W}$, and if
$I\cap C(X,t,r)\neq\emptyset$ for some $(X,t)\in\Sigma$, $r>0$, such that $C(X,t,r)\subset C(X_0,t_0,r_0)$, then $I^\ast\subset C(X_0,t_r,r_0/8)$. With $c$ fixed we introduce, for $(X,t)\in\Sigma$ and $r>0$ such that $C(X,t,r)\subset C(X_0,t_0,r_0)$,
\begin{align}\label{doma-}
\W(X,t,r)&:=\{I:\ I\cap C(X,t,r)\neq\emptyset\},
\end{align}
and
\begin{align}\label{doma}
\Omega(X,t,r)&:=\bigcup_{I:\ I\in \W(X,t,r)}I,\quad \Omega^\ast(X,t,r):=\bigcup_{I:\ I\in \W(X,t,r)}I^\ast.
\end{align}
 Given a large enough integer $N$, we also introduce
 \begin{align}\label{doma+-}
\W_N(X,t,r)&:=\{I:\ I\in \W(X,t,r),\ \ell(I)\geq 2^{-N}r\},
\end{align}
and
 \begin{align}\label{doma+}
\Omega_N(X,t,r)&:=\bigcup_{I:\ I\in \W_N(X,t,r)}I,\quad \Omega_N^\ast(X,t,r):=\bigcup_{I:\ I\in \W_N(X,t,r)}I^\ast.
\end{align}
Then clearly
\begin{align}\label{doma++}
\Omega(X,t,r)\subset \Omega^\ast(X,t,r),\ \Omega_N(X,t,r)\subset \Omega_N^\ast(X,t,r),
\end{align}
and
\begin{align}\label{doma++a}
 \Omega_N(X,t,r)\subset \Omega_{N'}(X,t,r),\ \Omega_N^\ast(X,t,r)\subset \Omega_{N'}^\ast(X,t,r),
\end{align}
whenever $N\leq N'$. Furthermore, this and the monotone convergence theorem imply that
\begin{equation}
\iiint_{\Omega(X,t,r)} G(Y,s)\,\d Y\d s
=\lim_{N\to\infty} \iiint_{\Omega_N(X,t,r)} G(Y,s)\,\d Y\d s,
\label{eq:Omega_N-TCM}
\end{equation}
for any integrable, non-negative function $G$ on $\Omega(X,t,r)$. The conclusion remain valid with $\Omega_N(X,t,r)$, $\Omega(X,t,r)$, replaced by
$\Omega_N^\ast(X,t,r)$, $\Omega^\ast(X,t,r)$. Finally, we introduce
 \begin{align}\label{doma+-ha}
\Gamma(X,t,r)&:=\{I\in \W(X,t,r):\ I^\ast\cap (\Omega\setminus C(X,t,r))\neq\emptyset\},\notag\\
\Gamma_N(X,t,r)&:=\{I\in\W(X,t,r):\ \exists I'\in\W\mbox{ such that }I^\ast\cap I'\neq\emptyset,\  \ell(I')\leq 2^{-N}r\},
\end{align}
and
\begin{align}\label{doma+-ha+}
I_N(X,t,r):=\W(X,t,r)\setminus \Gamma(X,t,r)\setminus \Gamma_N(X,t,r).
\end{align}

\begin{lemma}\label{lemma:approx-saw} Consider $(X,t)\in\Sigma$, $r>0$, let $N\gg 1$ and let  $\Omega_N:=\Omega_N(X,t,r)$,  $\Omega_N^\ast:=\Omega_N^\ast(X,t,r)$,  and $I_N:=I_N(X,t,r)$ be defined as above. Then there exists
$\Psi_N\in C_0^\infty(\mathbb R^{n+1})$ such that
\begin{align}\label{Coral}
\quad 1_{\Omega_N}\lesssim \Psi_N\le 1_{\Omega_N^\ast},
\end{align}
and such that
\begin{align}\label{Corb}
\quad |\nabla_X\Psi_N(Y,s)|\delta(Y,s)+|\partial_t\Psi_N(Y,s)|\delta^2(Y,s)\lesssim 1,
\end{align}
for all $(Y,s)\in\mathbb R^{n+1}\setminus\Sigma$. Furthermore,
\begin{equation}
|\nabla_X\Psi_N(Y,s)|+|\partial_t\Psi_N(Y,s)|\equiv 0\mbox{ whenever }(Y,s)\in\bigcup_{I\in I_N}I^{*}.
\label{eq:fregtgtr}
\end{equation}
\end{lemma}
\begin{proof} Given $I$, any closed dyadic parabolic cube in $\mathbb R^{n+1}$, we have $I^{*}=(1+2\tau)I$. We here also introduce $\tilde I=(1+\tau)I$ so that
\begin{equation}
I
\subsetneq
\interior(\tilde I)
\subsetneq \tilde I
\subset
\interior(I^{*}).
\label{eq:56y6y6}
\end{equation}
Given $I_0:=[-1/2,1/2]^{n}\times [-1/4,1/4]\subset\mathbb R^{n+1}$, fix $\phi_0\in C_0^\infty(\mathbb R^{n+1})$ such that
 $1_{I_0}\le \phi_0\le 1_{\tilde I_0}$ and $|\nabla_X \phi_0|+|\partial_t \phi_0|\lesssim 1$, here the implicit constant will also depend on the parameter $\tau$. For every $I\in \W=\W(\Omega)$, we set $$\phi_I(Y,s)=\phi_0((Y-X(I))/\ell(I),(s-t(I))/\ell^2(I)),$$
 where $(X(I),t(I))$ denotes the center of $I$. Then $\phi_I\in C^\infty(\mathbb R^{n+1})$, $1_{I}\le \phi_I\le 1_{\tilde I}$ and
  $$\ell(I)|\nabla_X \phi_I|+\ell^2(I)|\partial_t \phi_I|\lesssim 1.$$ To start the construction of $\Psi_N$ leading to \eqref{Coral}, we let, for every $(Y,s)\in\Omega=\mathbb R^{n+1}\setminus\Sigma$,  $$\phi(Y,s):=\sum_{I\in \W} \phi_I(Y,s).$$ It then follows that $\phi\in C_{\rm loc}^\infty(\Omega)$ since for every compact subset of $\Omega$ the previous sum has finitely many non-vanishing terms. Also, $1\le \phi(Y,s)\lesssim c_{\tau}$ for every $(Y,s)\in \Omega$ since the family $\{\tilde I\}_{I\in \W}$ has bounded overlap by our choice of $\tau$. Hence, letting $\eta_I:=\phi_I/\phi$ we see that $\eta_I\in C_0^\infty(\mathbb R^{n+1})$, $c_\tau^{-1}1_{I}\le \eta_I\le 1_{\tilde I}$, and that $$\ell(I)|\nabla_X \eta_I|+\ell^2(I)|\partial_t \eta_I|\lesssim 1.$$ Using this, and recalling the definition of $\W_N=\W_N(X,t,r)$ in \eqref{doma+-}, we set
\begin{align}\label{cuttoff}
\Psi_N(Y,s):=\sum_{I\in \W_N} \eta_I(Y,s)=
\frac{\sum\limits_{I\in \W_N} \phi_I(Y,s)}{\sum\limits_{I\in \W} \phi_I(Y,s)},
\end{align}
for all $(Y,s)\in\Omega$. Note that the number of terms in the sum defining $\Psi_N$ is bounded depending on $N$. This and the fact that each $\eta_I\in C_0^\infty(\mathbb R^{n+1})$ yield that $\Psi_N\in C_0^\infty(\mathbb R^{n+1})$. By construction
\begin{align*}
\supp \Psi_N
\subset \bigcup_{I\in \W_N} \tilde I=
\bigcup_{I\in \W_N(X,t,r)} \tilde I\subset
\Omega_N^\ast(X,t,r)=\Omega_N^\ast.
\end{align*}
This, the fact that $\W_N\subset \W$ and the definition of $\Psi_N$ immediately gives that
$\Psi_N\le 1_{\Omega_N^\ast(X,t,r)}$. On the other hand if $(Y,s)\in \Omega_N(X,t,r)$ then the exists $I\in \W_N$ such that $(Y,s)\in I$ in which case $\Psi_N(Y,s)\ge \eta_I(Y,s)\ge c_\tau^{-1}$. This completes the proof of \eqref{Coral}. To prove \eqref{Corb} we note that for every $(Y,s)\in \Omega$,
$$
|\nabla_X \Psi_N(Y,s)|
\le
\sum_{I\in \W_N} |\nabla_X\eta_I(Y,s)|
\lesssim
\sum_{I\in \W} \ell(I)^{-1}\,1_{\tilde I}(Y,s)
\lesssim
\delta(Y,s)^{-1},
$$
where we have used that if $(Y,s)\in \tilde I$, then $\delta(Y,s)\approx \ell(I)$,  and the fact that the family $\{\tilde I\}_{I\in \W}$ has bounded overlap. The estimate for $|\partial_t \Psi_N(Y,s)|$ proceeds analogously. To prove \eqref{eq:fregtgtr}, we fix $I\in I_N(X,t,r)$ and  $(Y,s)\in I^{*}$, and we set $\W_{Y,s}:=\{J\in \W: \phi_J(Y,s)\neq 0\}$. We first note that  $\W_{Y,s}\subset \W_N$. Indeed, if $\phi_J(Y,s)\neq 0$ then $(X,t)\in \widetilde{J^{*}}$.
Hence $(Y,s)\in I^{*}\cap J^{*}$ and our choice of $\tau$ gives that $\partial I$ meets $\partial J$, this in turn implies that $J\in \W_N(X,t,r)$ since $I\in I_N(X,t,r)$. All this yields
$$
\Psi_N(Y,s)
=
\frac{\sum\limits_{J\in \W_N(X,t,r)} \phi_J(Y,s)}{\sum\limits_{J\in \W} \phi_J(Y,s)}
=
\frac{\sum\limits_{J\in \W_N(X,t,r)\cap \W_{Y,s}} \phi_J(Y,s)}{\sum\limits_{J\in \W(X,t,r)\cap \W_{Y,s}} \phi_J(Y,s)}
=
\frac{\sum\limits_{J\in \W_N(X,t,r)\cap \W_{Y,s}} \phi_J(Y,s)}{\sum\limits_{J\in \W_N(X,t,r)\cap \W_{Y,s}} \phi_J(Y,s)}
=
1.
$$
Hence $\Psi_N\big|_{I^{*}}\equiv 1$  for every $I\in I_N(X,t,r)$. This and the fact that $\Psi_N\in C_0^\infty(\mathbb R^{n+1})$ immediately give that $|\nabla_X\Psi_N|+|\partial_t\Psi_N|\equiv 0$ in $\bigcup_{I\in I_N(X,t,r) }I^{*}$. This completes the proof of the lemma.\end{proof}

\begin{lemma}\label{partialADR}  Consider $(X,t)\in\Sigma$, $r>0$, let $N\gg 1$ and let $\Omega_N^\ast=\Omega_N^\ast(X,t,r)$ be defined as above.
 Let $\Psi=\Psi_N\in C_0^\infty(\mathbb R^{n+1})$ be as in Lemma \ref{lemma:approx-saw}. Then
\begin{align*}
\iiint_{\Omega_N^\ast(X,t,r)} (|\nabla_X\Psi_N(Y,s)|+\delta(Y,s)|\partial_t\Psi_N(Y,s)|)\, \d Y\d s\lesssim r^{n+1},\end{align*}
with implicit constant depending on the admissible constants of the construction, but not on $N$.
\end{lemma}
\begin{proof} Consider $(X,t)\in\Sigma$, $r>0$, let $N\gg 1$ and let  $\Gamma:=\Gamma(X,t,r)$, $\Gamma_N:=\Gamma_N(X,t,r)$, and $I_N:=I_N(X,t,r)$ be defined as above. Using Lemma \ref{lemma:approx-saw} it follows that
\begin{align*}
\iiint_{\Omega_N^\ast(X,t,r)} (|\nabla_X\Psi_N(Y,s)|+\delta(Y,s)|\partial_t\Psi_N(Y,s)|)\, \d Y\d s\lesssim \sum_{I\in \Gamma\cup\Gamma_N}\ell(I)^{n+1}.\end{align*}
Hence, to prove the lemma is suffices to prove that
\begin{equation}
\sum_{I\in \Gamma\cup\Gamma_N}\ell(I)^{n+1}\lesssim r^{n+1},
\label{eq:fregtgtr+}
\end{equation}
with implicit constant depending on the allowable parameters but being uniform in $N$. To estimate the contribution to sum in \eqref{eq:fregtgtr+} coming from $\Gamma_N$, we first note that if $I\in \Gamma_N(X,t,r)$, then $\ell(I)\approx 2^{-N}r$ and there exists $(Z_I,\tau_I)\in \Sigma$ such that
\begin{align}\label{bra}
\dist(Z_I,\tau_I,I)\sim \ell(I)\sim \delta(Z_I,\tau_I).
\end{align}
Using $( Z_I, \tau_I)\in\Sigma$ we let $\Delta_I:=\Sigma\cap C( Z_I, \tau_I,\delta(Z_I,\tau_I))$, which is a surface ball/cube on $\Sigma$ centered at $( Z_I, \tau_I)\in \Sigma$. Since $\Sigma$ is parabolic ADR  it follows that
\begin{equation}
\sum_{I\in \Gamma_N(X,t,r)}\ell(I)^{n+1}\lesssim \sum_{I\in \Gamma_N(X,t,r)}\sigma(\Delta_I).
\label{eq:fregtgtr+blla+}
\end{equation}
We now simply note that the members of the family $\{\Delta_I\}_{I\in \Gamma_N(X,t,r)}$ have bounded overlap. Hence
\begin{equation}
\sum_{I\in \Gamma_N(X,t,r)}\ell(I)^{n+1}\lesssim r^{n+1}.
\label{eq:fregtgtr+blla++}
\end{equation}
To handle the contribution to \eqref{eq:fregtgtr+} from $I\in \Gamma(X,t,r)$ we note that $\partial C(X,t,r)=S\cup T\cup B$ where $S$ is the lateral part of the cylinder, and $T$ and $B$ are the top and the bottom of the cylinder, respectively. We introduce
\begin{align*}
\Gamma_{T}(X,t,r)&:=\{I\in \Gamma(X,t,r):\ I^\ast\cap T\neq \emptyset\},\\
\Gamma_{B}(X,t,r)&:=\{I\in \Gamma(X,t,r):\ I^\ast\cap B\neq \emptyset\},
\end{align*}
and $\Gamma_{S}(X,t,r)=\Gamma(X,t,r)\setminus\Gamma_{T}(X,t,r)\setminus\Gamma_{B}(X,t,r)$.  Hence,
\begin{equation}
\sum_{I\in \Gamma(X,t,r)}\ell(I)^{n+1}\leq \sum_{I\in \Gamma_S(X,t,r)}\ell(I)^{n+1}+\sum_{I\in \Gamma_T(X,t,r)}\ell(I)^{n+1}+\sum_{I\in \Gamma_B(X,t,r)}\ell(I)^{n+1}.
\label{eq:fregtgtr+blla+ML}
\end{equation}
Given $I\in \Gamma_S(X,t,r)$ we see that there exists a cylinder $C_I=C_I(\hat X_I,\hat t_I,\hat r_I)$ such that $(\hat X_I,\hat t_I)\in S$,
$\hat r_I\sim \ell(I)$, and such that $\mathcal{H}_p^{n+1}(S\cap C_I)\sim \ell(I)^{n+1}$. The cylinders $\{C_I\}$ have bounded overlap, as the Whitney cubes have, and it follows that
\begin{equation}
\sum_{I\in \Gamma_S(X,t,r)}\ell(I)^{n+1}\lesssim \sum_{I\in \Gamma_S(X,t,r)} \mathcal{H}_p^{n+1}(S\cap C_I)\lesssim \mathcal{H}_p^{n+1}(S)\lesssim r^{n+1}.
\label{eq:fregtgtr+blla+ML+}
\end{equation}
Similarly, given $I\in \Gamma_B(X,t,r)$ we see that there exists a cylinder $C_I=C_I(\hat X_I,\hat t_I,\hat r_I)$ such that $(\hat X_I,\hat t_I)\in B$,
$\hat r_I\sim \ell(I)$, and such that $|B\cap C_I|\sim \ell(I)^{n}$. Again, the cylinders $\{C_I\}$ have bounded overlap and we see that
\begin{equation}
\sum_{I\in \Gamma_B(X,t,r)}\ell(I)^{n+1}\lesssim r\sum_{I\in \Gamma_B(X,t,r)}\ell(I)^{n}\lesssim r\sum_{I\in \Gamma_B(X,t,r)}|B\cap C_I|
\lesssim r|B|\lesssim r^{n+1}.
\label{eq:fregtgtr+blla+ML++}
\end{equation}
The sum with respect to $I\in \Gamma_T(X,t,r)$ can be treated analogously.  This completes the proof of \eqref{eq:fregtgtr+} and the proof of the lemma.
\end{proof}

\section{Proof of Theorem \ref{thm1-a}}\label{Proofmain}

Let $(X,t)\in \Sigma$, $r>0$, $C(X,t,r)\subset C(X_0,t_0,r_0/c)$, $\W_N=\W_N(X,t,r)$, and recall the cutoff introduced in \eqref{cuttoff}, i.e.,
\begin{align*}
\Psi_N(Y,s)=\sum_{I\in \W_N} \eta_I(Y,s).
\end{align*}
Based on this we introduce
\begin{align}\label{sq3}
J:=\iiint |\nabla_X u|^{q}(u_{x_ix_j})^2u\eta_I\, \d Y\d s,\ \tilde J:=\frac 1 2 \iiint |\nabla_X u|^{r}u_{t}^2u\eta_I\, \d Y\d s,
\end{align}
where we use summation convention with respect to  $I\in \W_N$.  \eqref{boundsaapall} implies  that $u$ is a strong/classical solution to the evolutionary $p$-Laplacian in $\Omega\cap C(X_0,t_0,2r_0)$, see subsection \ref{nondd}. In particular, in $\Omega\cap C(X_0,t_0,2r_0)$
$u$ satisfies, in the strong or pointwise sense,
\begin{align}\label{sq1}
u_t=\nabla_X\cdot (|\nabla_X u|^{p-2}\nabla_X u)=|\nabla_X u|^{p-2}\Delta u+(p-2)|\nabla_X u|^{p-4}\Delta_\infty u,
\end{align}
where
\begin{align}\label{sq2}
\Delta_\infty u:=u_{x_i}u_{x_j}u_{x_ix_j}.
\end{align}
To prove Theorem \ref{thm1-a} we will integrate by parts a number of times in $J$ and $\tilde J$. Each integration by parts will result in a number of terms. A term labelled with the letter $G$
will be referred to as a Good term and such a term will either contain a factor $(\eta_I)_{x_j}$ or a factor $(\eta_I)_{t}$. Using this, \eqref{boundsaapa} and Lemma \ref{partialADR}, we will always be able to conclude that such a term satisfies $|G|\lesssim r^{n+1}$, with implicit constant depending only  on $n$, $M$, $\gamma$, $p$, $q$, and $r$. Later we choose $r=q-2p+4$.

To start the proof we first estimate $J$ and in this case we first note, using integration by parts, that
\begin{align}\label{sq4}
J&=\iiint|\nabla_X u|^{q}u_{x_ix_j}u_{x_ix_j}u\eta_I\, \d Y\d s=-G_1-J_1-J_2-qJ_3,
\end{align}
where
\begin{align}\label{sq4b}
G_1&:=\iiint|\nabla_X u|^{q}u_{x_ix_j}u_{x_i}u(\eta_I)_{x_j} \, \d Y\d s,\notag\\
J_1&:=\iiint|\nabla_X u|^{q}(\Delta_\infty u)\eta_I\, \d Y\d s,\notag\\
J_2&:=\iiint|\nabla_X u|^{q}(\Delta u_{x_i})u_{x_i}u\eta_I\, \d Y\d s,\notag\\
J_3&:=\iiint|\nabla_X u|^{q}u_{x_k}u_{x_kx_j}u_{x_ix_j}u_{x_i}u\eta_I\, \d Y\d s.
\end{align}
Note that
\begin{align}\label{sq5}
J_3=\frac 1 4\iiint|\nabla_X u|^{q}((|\nabla_X u|^2)_{x_j})^2u\eta_I\, \d Y\d s>0.
\end{align}
Hence,
\begin{align}\label{sq4+}
J+qJ_3=-G_1-J_1-J_2.
\end{align}
We  manipulate $J_2$ by using the equation. Note that
\begin{align}\label{sq6}
\Delta u=|\nabla_X u|^{2-p}u_t-(p-2)|\nabla_X u|^{-2}\Delta_\infty u.
\end{align}
Hence,
\begin{align}\label{sq7}
\Delta u_{x_i}=&(|\nabla_X u|^{2-p}u_t)_{x_i}-(p-2)(|\nabla_X u|^{-2}\Delta_\infty u)_{x_i}\notag\\
=&|\nabla_X u|^{2-p}u_{tx_i}+(2-p)|\nabla_X u|^{-p}u_{x_k}u_{x_kx_i}u_t\notag\\
&+2(p-2)|\nabla_X u|^{-4}u_{x_k}u_{x_kx_i}\Delta_\infty u-(p-2)|\nabla_X u|^{-2}(\Delta_\infty u)_{x_i}.
\end{align}
Using this we see that
\begin{align}\label{sq8}
J_2=&\iiint|\nabla_X u|^{q+2-p}u_{tx_i}u_{x_i}u\eta_I\, \d Y\d s\notag\\
&+(2-p)\iiint|\nabla_X u|^{q-p}u_{x_k}u_{x_kx_i}u_tu_{x_i}u\eta_I\, \d Y\d s\notag\\
&+2(p-2)\iiint|\nabla_X u|^{q-4}u_{x_k}u_{x_kx_i}(\Delta_\infty u) u_{x_i}u\eta_I\, \d Y\d s\notag\\
&-(p-2)\iiint|\nabla_X u|^{q-2}(\Delta_\infty u)_{x_i}u_{x_i}u\eta_I\, \d Y\d s .
\end{align}
This can be simplified and we deduce
\begin{align}\label{sq9}
J_2=&\iiint|\nabla_X u|^{q+2-p}u_{tx_i}u_{x_i}u\eta_I\, \d Y\d s\notag\\
&-(p-2)\iiint|\nabla_X u|^{q-p}u_t(\Delta_\infty u)u\eta_I\, \d Y\d s\notag\\
&+2(p-2)\iiint|\nabla_X u|^{q-4}(\Delta_\infty u)^2 u\eta_I\, \d Y\d s\notag\\
&-(p-2)\iiint|\nabla_X u|^{q-2}(\Delta_\infty u)_{x_i}u_{x_i}u\eta_I\, \d Y\d s.
\end{align}
Let
\begin{align}\label{sq9aa}
J_{21}=\iiint|\nabla_X u|^{q+2-p}u_{tx_i}u_{x_i}u\eta_I\, \d Y\d s.
\end{align}
Then
\begin{align}\label{sq10}
J_{21}=&\frac 1 2\iiint|\nabla_X u|^{q+2-p}(|\nabla_X u|^2)_tu\eta_I\, \d Y\d s\notag\\
=&-\frac 1 2\iiint|\nabla_X u|^{q+4-p}u(\eta_I)_t\, \d Y\d s\notag\\
&-\frac1 2 \iiint|\nabla_X u|^{q+4-p}u_t\eta_I\, \d Y\d s\notag\\
&-\frac 1 2(q+2-p)\iiint|\nabla_X u|^{q+2-p}u_{tx_i}u_{x_i}u\eta_I\, \d Y\d s.
\end{align}
In particular, we can conclude that
\begin{align}\label{sq11}
J_{21}=&-\frac 1 {(q+4-p)}\iiint|\nabla_X u|^{q+4-p}u(\eta_I)_t\, \d Y\d s\notag\\
&-\frac1 {(q+4-p)} \iiint|\nabla_X u|^{q+4-p}u_t\eta_I\, \d Y\d s.
\end{align}
Put together we see that
\begin{align}\label{sq9-}
-J_1-J_2=&G_2-\iiint|\nabla_X u|^{q}(\Delta_\infty u)\eta_I\, \d Y\d s\notag\\
&+\frac1 {(q+4-p)} \iiint|\nabla_X u|^{q+4-p}u_t\eta_I\, \d Y\d s\notag\\
&+(p-2)\iiint|\nabla_X u|^{q-p}u_t(\Delta_\infty u)u\eta_I\, \d Y\d s\notag\\
&-2(p-2)\iiint|\nabla_X u|^{q-4}(\Delta_\infty u)^2 u\eta_I\, \d Y\d s\notag\\
&+(p-2)\iiint|\nabla_X u|^{q-2}(\Delta_\infty u)_{x_i}u_{x_i}u\eta_I\, \d Y\d s,
\end{align}
where
\begin{align}\label{sq10Bonn}
G_{2}:=\frac 1 {(q+4-p)}\iiint|\nabla_X u|^{q+4-p}u(\eta_I)_t\, \d Y\d s.
\end{align}
Using $u_t=|\nabla_X u|^{p-2}\Delta u+(p-2)|\nabla_X u|^{p-4}\Delta_\infty u$ in \eqref{sq9-} we deduce
\begin{align}\label{sq9-+}
-J_1-J_2=&G_2-\iiint|\nabla_X u|^{q}(\Delta_\infty u)\eta_I\, \d Y\d s\notag\\
&+\frac1 {(q+4-p)} \iiint|\nabla_X u|^{q+2}(\Delta u)\eta_I\, \d Y\d s\notag\\
&+\frac{(p-2)} {(q+4-p)} \iiint|\nabla_X u|^{q}(\Delta_\infty u)\eta_I\, \d Y\d s\notag\\
&+(p-2)\iiint|\nabla_X u|^{q-2}(\Delta u)(\Delta_\infty u)u\eta_I\, \d Y\d s\notag\\
&+(p-2)(p-4)\iiint|\nabla_X u|^{q-4}(\Delta_\infty u)^2u\eta_I\, \d Y\d s\notag\\
&+(p-2)\iiint|\nabla_X u|^{q-2}(\Delta_\infty u)_{x_i}u_{x_i}u\eta_I\, \d Y\d s.
\end{align}

We next manipulate $\tilde J$ and in this case we first note that
\begin{align}\label{sq1++}
\frac 1 2(u_t)^2\leq (|\nabla_X u|^{2p-4}(\Delta u)^2+(p-2)^2|\nabla_X u|^{2p-8}(\Delta_\infty u)^2.
\end{align}
Hence,
\begin{align}\label{sq3a}
\tilde J\leq& \tilde J_1+\tilde J_2,
\end{align}
where
\begin{align}\label{sq3abonn}
\tilde J_1&:=\iiint|\nabla_X u|^{r+2p-4}(\Delta u)^2u\eta_I\, \d Y\d s,\notag\\
\tilde J_2&:=(p-2)^2\iiint|\nabla_X u|^{r+2p-8}(\Delta_\infty u)^2u\eta_I\, \d Y\d s.
\end{align}
Using partial integration we see that
\begin{align}\label{sq11ml}
\tilde J_1=&-\iiint|\nabla_X u|^{r+2p-4}(\Delta u)u_{x_i}u(\eta_I)_{x_i} \, \d Y\d s\notag\\
&-\iiint|\nabla_X u|^{r+2p-2}(\Delta u)\eta_I\, \d Y\d s\notag\\
&-\iiint|\nabla_X u|^{r+2p-4}(\Delta u_{x_i})u_{x_i}u\eta_I\, \d Y\d s\notag\\
&-(r+2p-4)\iiint|\nabla_X u|^{r+2p-6}(\Delta u)(\Delta_\infty u)u\eta_I\, \d Y\d s.
\end{align}
Furthermore,
\begin{align}\label{sq11++}
&\iiint|\nabla_X u|^{r+2p-6}(\Delta u)(\Delta_\infty u)u\eta_I\, \d Y\d s\notag\\
=&\iiint|\nabla_X u|^{r+2p-6}u_{x_ix_i}u_{x_k}u_{x_kx_j}u_{x_j}u\eta_I\, \d Y\d s\notag\\
=&-\iiint|\nabla_X u|^{r+2p-4}u_{x_ix_i}u_{x_j}u(\eta_I)_{x_j} \, \d Y\d s-\iiint|\nabla_X u|^{r+2p-2}(\Delta u)\eta_I\, \d Y\d s\notag\\
&-\iiint|\nabla_X u|^{r+2p-4}(\Delta u)^2u\eta_I\, \d Y\d s-\iiint|\nabla_X u|^{r+2p-4}(\Delta u_{x_j})u_{x_j}\eta_I\, \d Y\d s.
\end{align}
Combining the last two displays we see that
\begin{align}\label{sq11j}
\tilde J_1=&-G_3-\iiint|\nabla_X u|^{r+2p-2}(\Delta u)\eta_I\, \d Y\d s\notag\\
&-\iiint|\nabla_X u|^{r+2p-4}(\Delta u_{x_i})u_{x_i}u\eta_I\, \d Y\d s,
\end{align}
where
\begin{align}\label{sq11bon}
G_3:=\iiint|\nabla_X u|^{r+2p-4}(\Delta u)u_{x_i}u(\eta_I)_{x_i} \, \d Y\d s+\iiint|\nabla_X u|^{r+2p-4}u_{x_ix_i}u_{x_j}u(\eta_I)_{x_j} \, \d Y\d s.
\end{align}
Hence,
\begin{align}\label{sq3aml}
\tilde J\leq &-G_3-\iiint|\nabla_X u|^{r+2p-2}(\Delta u)\eta_I\, \d Y\d s\notag\\
&-\iiint|\nabla_X u|^{r+2p-4}(\Delta u_{x_i})u_{x_i}u\eta_I\, \d Y\d s\notag\\
&+(p-2)^2\iiint|\nabla_X u|^{r+2p-8}(\Delta_\infty u)^2u\eta_I\, \d Y\d s.
\end{align}
Repeating the deduction made for $J_2$ we see that
\begin{align}\label{sq11+ml}
&-\iiint|\nabla_X u|^{r+2p-4}(\Delta u_{x_i})u_{x_i}u\eta_I\, \d Y\d s\notag\\
&=\frac 1 {(r+p)}\iiint|\nabla_X u|^{q+p}u(\eta_I)_t\, \d Y\d s\notag\\
&+\frac1 {(r+p)} \iiint|\nabla_X u|^{r+p}u_t\eta_I\, \d Y\d s\notag\\
&+(p-2)\iiint|\nabla_X u|^{r+p-4}u_t(\Delta_\infty u)u\eta_I\, \d Y\d s\notag\\
&-2(p-2)\iiint|\nabla_X u|^{r+2p-8}(\Delta_\infty u)^2 u\eta_I\, \d Y\d s\notag\\
&+(p-2)\iiint|\nabla_X u|^{r+2p-6}(\Delta_\infty u)_{x_i}u_{x_i}u\eta_I\, \d Y\d s.
\end{align}
Put together we have
\begin{align}\label{sq11+}
\tilde J\leq&-G_3+G_4-\iiint|\nabla_X u|^{r+2p-2}(\Delta u)\eta_I\, \d Y\d s\notag\\
&+\frac1 {(r+p)} \iiint|\nabla_X u|^{r+p}u_t\eta_I\, \d Y\d s\notag\\
&+(p-2)\iiint|\nabla_X u|^{r+p-4}u_t(\Delta_\infty u)u\eta_I\, \d Y\d s\notag\\
&+2(p-2)(p-4)\iiint|\nabla_X u|^{r+2p-8}(\Delta_\infty u)^2 u\eta_I\, \d Y\d s\notag\\
&+(p-2)\iiint|\nabla_X u|^{r+2p-6}(\Delta_\infty u)_{x_i}u_{x_i}u\eta_I\, \d Y\d s,
\end{align}
where
\begin{align}\label{sq11+bon}
G_4:=\frac 1 {(r+p)}\iiint|\nabla_X u|^{q+p}u(\eta_I)_t\, \d Y\d s.
\end{align}
Using that $u_t=|\nabla_X u|^{p-2}\Delta u+(p-2)|\nabla_X u|^{p-4}\Delta_\infty u$, \eqref{sq11+} implies that
\begin{align}\label{sq11+a}
\tilde J\leq& -G_3+G_4-\iiint|\nabla_X u|^{r+2p-2}(\Delta u)\eta_I\, \d Y\d s\notag\\
&+(p-2)\iiint|\nabla_X u|^{r+2p-6}(\Delta u)(\Delta_\infty u)u\eta_I\, \d Y\d s\notag\\
&+\frac1 {(r+p)} \iiint|\nabla_X u|^{r+2p-2}(\Delta u)\eta_I\, \d Y\d s\notag\\
&+\frac{(p-2)} {(r+p)} \iiint|\nabla_X u|^{r+2p-4}(\Delta_\infty u)\eta_I\, \d Y\d s\notag\\
&+(p-2)(3p-10)\iiint|\nabla_X u|^{r+2p-8}(\Delta_\infty u)^2u\eta_I\, \d Y\d s\notag\\
&+(p-2)\iiint|\nabla_X u|^{r+2p-6}(\Delta_\infty u)_{x_i}u_{x_i}u\eta_I\, \d Y\d s.
\end{align}

Above we have conducted some preliminary manipulations of $J$ and $\tilde J$. In particular, by \eqref{sq4+},
\begin{align*}
J+qJ_3=-G_1-J_1-J_2,
\end{align*}
and in \eqref{sq9-+} we concluded an equality for $-J_1-J_2$. We have also deduced an upper bound on $\tilde J$ in \eqref{sq11+a}.

To proceed, we introduce, for $a>0$,
\begin{align}\label{sq11+b-}
T^a&:=\iiint|\nabla_X u|^{a}(\Delta u)(\Delta_\infty u)u\eta_I\, \d Y\d s\notag\\
&+\iiint|\nabla_X u|^{a}(\Delta_\infty u)_{x_i}u_{x_i}u\eta_I\, \d Y\d s.
\end{align}
Integrating by parts in the second term we see that
\begin{align}\label{sq11+b}
T^a=&-G_5- \iiint|\nabla_X u|^{a+2}(\Delta_\infty u)\eta_I\, \d Y\d s\notag\\
&-a\iiint|\nabla_X u|^{a-2}(\Delta_\infty u)^2u\eta_I\, \d Y\d s,
\end{align}
where
\begin{align}\label{sq11+ba}
G_5&:=\iiint|\nabla_X u|^{a}(\Delta_\infty u)u_{x_i}u(\eta_I)_{x_i}\, \d Y\d s.
\end{align}

Using \eqref{sq11+b} in \eqref{sq9-+} with $a={q-2}$,
\begin{align}\label{sq9-ffa}
-J_1-J_2=&-G_1-G_5-\biggl ((p-1)-\frac{(p-2)} {(q+4-p)}\biggr )\iiint|\nabla_X u|^{q}(\Delta_\infty u)\eta_I\, \d Y\d s\notag\\
&+\frac1 {(q+4-p)} \iiint|\nabla_X u|^{q+2}(\Delta u)\eta_I\, \d Y\d s\notag\\
&+(p-2)(p-q-2)\iiint|\nabla_X u|^{q-4}(\Delta_\infty u)^2u\eta_I\, \d Y\d s.
\end{align}
Similarly, using \eqref{sq11+b} with $a=r+2p-6$ in \eqref{sq11+a} we see that
\begin{align}\label{sq11+ajy}
\tilde J\leq& -G_3+G_4-G_5-\biggl (1-\frac1 {(r+p)}\biggr )\iiint|\nabla_X u|^{r+2p-2}(\Delta u)\eta_I\, \d Y\d s\notag\\
&-\biggl ((p-2)-\frac{(p-2)} {(r+p)}\biggr )\iiint|\nabla_X u|^{r+2p-4}(\Delta_\infty u)\eta_I\, \d Y\d s\notag\\
&+(p-2)(p-r-4)\iiint|\nabla_X u|^{r+2p-8}(\Delta_\infty u)^2u\eta_I\, \d Y\d s.
\end{align}

Next we introduce
\begin{align}\label{sq11+ajya}
\tilde T^a:=\iiint|\nabla_X u|^{a}(\Delta_\infty u)\eta_I\, \d Y\d s=\iiint|\nabla_X u|^{a}u_{x_i}u_{x_ix_j}u_{x_j}\eta_I\, \d Y\d s.
\end{align}
Again by partial integration
\begin{align}\label{sq11+ajya+}
\tilde T^a=G_6-\iiint|\nabla_X u|^{a+2}(\Delta u)\eta_I\, \d Y\d s-(a+1)\tilde T_a,
\end{align}
where
\begin{align}\label{sq11+ajya+h}
G_6:=\iiint|\nabla_X u|^{a}u_{x_i}u_{x_i}u_{x_j}(\eta_I)_{x_j}\, \d Y\d s.
\end{align}
Hence,
\begin{align}\label{sq11+ajya+a}
\tilde T^a=\frac {G_6}{a+2}-\frac 1{a+2}\iiint|\nabla_X u|^{a+2}(\Delta u)\eta_I\, \d Y\d s.
\end{align}
Using this we see that \eqref{sq9-ffa} simplifies to
\begin{align}\label{sq9-ffaffa}
-J_1-J_2=& G+\biggl (\frac1 {(q+4-p)}+\frac{(p-1)}{q+2} +\frac{(p-2)} {(q+4-p)}\frac 1{q+2}\biggr )\iiint|\nabla_X u|^{q+2}(\Delta u)\eta_I\, \d Y\d s\notag\\
&+(p-2)(p-q-2)\iiint|\nabla_X u|^{q-4}(\Delta_\infty u)^2u\eta_I\, \d Y\d s,
\end{align}
and that \eqref{sq11+ajy} simplifies to
\begin{align}\label{sq11+ajyk}
\tilde J\leq& G-\biggl (1-\frac1 {(r+p)}\biggr )\biggl (1-\frac {(p-2)}{(r+2p-2)}\biggr )\iiint|\nabla_X u|^{r+2p-2}(\Delta u)\eta_I\, \d Y\d s\notag\\
&+(p-2)(p-r-4)\iiint|\nabla_X u|^{r+2p-8}(\Delta_\infty u)^2u\eta_I\, \d Y\d s.
\end{align}
In the last two displays, and from now on, $G$ is a Good term which depends linearly on $G_1-G_6$.

We now consider the sum $J+qJ_3+\alpha \tilde J$ where $\alpha>0$ is a constant to be chosen in order to achieve cancellation. We have
\begin{align}\label{sq11+ajykpp}
&J+qJ_3+\alpha \tilde J\notag\\
\leq&G+\biggl (\frac1 {(q+4-p)}+\frac{(p-1)}{q+2} +\frac{(p-2)} {(q+4-p)}\frac 1{q+2}\biggr )\iiint|\nabla_X u|^{q+2}(\Delta u)\eta_I\, \d Y\d s\notag\\
&+(p-2)(p-q-2)\iiint|\nabla_X u|^{q-4}(\Delta_\infty u)^2u\eta_I\, \d Y\d s\notag\\
&-\alpha\biggl (1-\frac1 {(r+p)}\biggr )\biggl (1-\frac {(p-2)}{(r+2p-2)}\biggr )\iiint|\nabla_X u|^{r+2p-2}(\Delta u)\eta_I\, \d Y\d s\notag\\
&+\alpha(p-2)(p-r-4)\iiint|\nabla_X u|^{r+2p-8}(\Delta_\infty u)^2u\eta_I\, \d Y\d s.
\end{align}
Let $r=q-2p+4$ and
\begin{align}\label{sq11+ajykppjj}
E_1&:=\iiint|\nabla_X u|^{q+2}(\Delta u)\eta_I\, \d Y\d s,\notag\\
E_2&:=\iiint|\nabla_X u|^{q-4}(\Delta_\infty u)^2u\eta_I\, \d Y\d s.
\end{align}
Then \eqref{sq11+ajykpp} can be expressed
\begin{align}\label{sq11+ajykppagag}
J+qJ_3+\alpha \tilde J \leq& G+\biggl (\frac1 {(q-p+4)}+\frac{(p-1)}{q+2} +\frac{(p-2)} {(q-p+4)}\frac 1{q+2}\biggr )E_1\notag\\
&-\alpha\biggl (1-\frac1 {(q-p+4)}\biggr )\biggl (1-\frac {(p-2)}{(q+2)}\biggr )E_1\notag\\
&+(p-2)\biggl((p-q-2)+\alpha(p-q-8)\biggr)E_2.
\end{align}
Given $q$ we let $\alpha$ solve the equation
\begin{align}\label{sq11+ajykppagaglll}
&\biggl (\frac1 {(q-p+4)}+\frac{(p-1)}{q+2} +\frac{(p-2)} {(q-p+4)}\frac 1{q+2}\biggr )=\alpha\biggl (1-\frac1 {(q-p+4)}\biggr )\biggl (1-\frac {(p-2)}{(q+2)}\biggr ).
\end{align}
This equation is equivalent to
\begin{align}\label{sq11+ajykppagagllllak}
&\biggl (\frac{p+q+(p-1)(q-p+4)}{(q-p+4)}\biggr )=\alpha(q-p+3).
\end{align}
Hence,
\begin{align}\label{sq11+ajykppagagllli}
\alpha =&\frac{p+q+(p-1)(q-p+4)}{(q-p+4)(q-p+3)}.
\end{align}
Note that $\alpha\in (0,\infty)$ if $q>p-4$. With this choose of $\alpha$, \eqref{sq11+ajykppagag} implies that
\begin{align}\label{sq11+ajykppagagbon}
J+qJ_3+\alpha \tilde J\leq G+(p-2)\biggl((p-q-2)+\alpha(p-q-8)\biggr)E_2.
\end{align}
Finally, if we let  $q>p-4$ be such that
\begin{align}\label{sq11+ajykppagagbon+ha}
-\beta:=(p-q-2)+\alpha(p-q-8)<0,
\end{align}
then
\begin{align}\label{sq11+ajykppagagbon+}
J+qJ_3+\alpha \tilde J+\beta(p-2)E_2\leq G.
\end{align}
Note that \eqref{sq11+ajykppagagbon+ha} is satisfied if and only if
\begin{align}\label{sq11+ajykppagagbon+ha+}
q>-\frac {2+8\alpha}{1+\alpha}+p=p-8+\frac 6{1+\alpha},
\end{align}
and obviously this is the case if $q\geq p-2$. The proof of Theorem \ref{thm1-a} is now almost complete. We just have to make sure that the terms $|\nabla_Xu|^\gamma$ appearing in all terms labelled with $G$ has an exponent $\gamma$ which is non-negative. To achieve this it is sufficient that $q\geq\max\{2,p-4\}$. We can conclude that
if $p\geq 4$, then $q\geq p-2$  works for all purposes, and if $4> p>2$, then $q\geq 2$ works for all purposes. In particular, if $q\geq p$ then all these cases are covered.

\begin{remark}\label{squarefunctiona}  Note that in  the proof of Theorem \ref{thm1-a} we actually establish several additional square function estimates that may be of interest. Indeed, from the proof we can conclude the following. Under the assumptions in the statement of Theorem \ref{thm1-a}, we have
\begin{align*}
 \iiint_{\Omega\cap C(X,t,r)}
 |\nabla_X u(Y,s)|^{q}((|\nabla_X u(Y,s)|^2)_{x_j})^2u(Y,s) \, \d Y\d s &\lesssim r^{n+1},\\
\iiint_{\Omega\cap C(X,t,r)}
 |\nabla_X u(Y,s)|^{q-4}|\Delta_\infty u(Y,s)|^2u(Y,s) \, \d Y\d s &\lesssim r^{n+1},\end{align*}
 where $\Delta_\infty u:=u_{x_i}u_{x_j}u_{x_ix_j}$, and with implicit constants depending only  on $n$, $M$, $\gamma$, $p$ and $q$.
\end{remark}

\subsection{Proof of Corollary \ref{thm1-b}}  Using the same notation as in the proof of Theorem \ref{thm1-a} we introduce
\begin{align}\label{sq3ml}
&K:=\iiint|\nabla_X u|^{a}(u_{x_ix_jx_l})^2u^3\eta_I\, \d Y\d s,\ L:=\frac 1 2 \iiint|\nabla_X u|^{b}(u_{tx_k})^2u^3\eta_I\, \d Y\d s,
\end{align}
where we again use summation convention with resepct to $I\in\W_N$, and where $a,b\in\mathbb R$.  We want to estimate $K$, $L$. As in the  proof of Theorem \ref{thm1-a}, we will integrate by parts a number of times in $K$ and $L$. Each integration by parts will result in a number of terms. A term labelled with the letter $G$
will be referred to as a Good term and such a term will either contain a factor $(\eta_I)_{x_j}$ or a factor $(\eta_I)_{t}$. Using this, \eqref{boundsaapa} and \eqref{boundsaapacor} and Lemma \ref{partialADR}, we will always be able to conclude that $|G|\lesssim r^{n+1}$.

To estimate $K$ we write
\begin{align}\label{sq4L--}
K=&\iiint|\nabla_X u|^{a}(u_{x_ix_jx_l})^2u^3\eta_I\, \d Y\d s=\iiint|\nabla_X u|^{a}u_{x_ix_jx_l}u_{x_ix_jx_l}u^3\eta_I\, \d Y\d s,
\end{align}
we perform integration by parts, and as a result
\begin{align}\label{sq4L}
K=&-\iiint|\nabla_X u|^{a}u_{x_ix_jx_l}u_{x_ix_j}u^3(\eta_I)_{x_l} \, \d Y\d s\notag\\
&-3\iiint|\nabla_X u|^{a}u_{x_ix_jx_l}u_{x_ix_j}u_{x_l}u^2\eta_I\, \d Y\d s\notag\\
&-\iiint|\nabla_X u|^{a}(\Delta u_{x_ix_j})u_{x_ix_j}u^3\eta_I\, \d Y\d s\notag\\
&-a\iiint|\nabla_X u|^{a}u_{x_k}u_{x_kx_l}u_{x_ix_jx_l}u_{x_ix_j}u^3\eta_I\, \d Y\d s\notag\\
&:=G-K_1-K_2-aK_3.
\end{align}
Using Cauchy-Schwarz with $\epsilon$, Theorem \ref{thm1-a}, \eqref{boundsaapa} and \eqref{boundsaapacor},
\begin{align}\label{sq4L+}
|K_1|+|K_3|\leq&  \epsilon K+c(\epsilon)r^{n+1}.
\end{align}
 Hence,
\begin{align}\label{sq4Lk}
K\lesssim &|G|+|K_2|=G+\biggl|\iiint|\nabla_X u|^{a}(\Delta u_{x_ix_j})u_{x_ix_j}u^3\eta_I\, \d Y\d s\biggr |.
\end{align}
Using the equation we see that
\begin{align}\label{sq7ag+}
\Delta u_{x_ix_j}=&|\nabla_X u|^{2-p}u_{tx_ix_j}+(2-p)|\nabla_X u|^{-p}u_{tx_i}u_{x_l}u_{x_lx_j}\notag\\
&+(2-p)|\nabla_X u|^{-p}u_{x_k}u_{x_kx_i}u_t+(2-p)|\nabla_X u|^{-p}u_{x_k}u_{x_kx_i}u_{tx_j}\notag\\
&+(2-p)|\nabla_X u|^{-p}u_{x_k}u_{x_kx_ix_j}u_t+(2-p)|\nabla_X u|^{-p}u_{x_kx_j}u_{x_kx_i}u_t\notag\\
&-(2-p)p|\nabla_X u|^{-p-2}u_{x_k}u_{x_kx_i}u_{x_l}u_{x_lx_j}u_t\notag\\
&+2(p-2)|\nabla_X u|^{-4}u_{x_k}u_{x_kx_i}(\Delta_\infty u)_{x_j}\notag\\
&+2(p-2)|\nabla_X u|^{-4}u_{x_k}u_{x_kx_ix_j}\Delta_\infty u\notag\\
&+2(p-2)|\nabla_X u|^{-4}u_{x_kx_j}u_{x_kx_i}\Delta_\infty u\notag\\
&-8(p-2)|\nabla_X u|^{-6}u_{x_l}u_{x_lx_j}u_{x_k}u_{x_kx_i}\Delta_\infty u\notag\\
&-(p-2)|\nabla_X u|^{-2}(\Delta_\infty u)_{x_ix_j}+2(p-2)|\nabla_X u|^{-4}u_{x_l}u_{x_lx_j}(\Delta_\infty u)_{x_i}.
\end{align}
Replacing $\Delta u_{x_ix_j}$ by the expression in \eqref{sq7ag+}, and again using Cauchy-Schwarz with $\epsilon$, \eqref{boundsaapa}, \eqref{boundsaapacor} and
Theorem \ref{thm1-a}, we see that
\begin{align}\label{sq4Lkhhae}
K_2=&\iiint|\nabla_X u|^{a}(\Delta u_{x_ix_j})u_{x_ix_j}u^3\eta_I\, \d Y\d s\notag\\
=&G+\frac 1 2 \iiint|\nabla_X u|^{a+2-p}((u_{x_ix_j})^2)_tu^3\eta_I\, \d Y\d s\notag\\
&-(p-2)\iiint|\nabla_X u|^{a-2}(\Delta_\infty u)_{x_ix_j}u_{x_ix_j}u^3\eta_I\, \d Y\d s\notag\\
=:&G+K_{21}+K_{22}.
\end{align}
Again $K_{21}$ can be treated as $G$. Focusing on $K_{22}$ we see that
\begin{align}\label{sq4Lkhhaf}
K_{22}=&-(p-2)\iiint|\nabla_X u|^{a-2}(\Delta_\infty u)_{x_ix_j}u_{x_ix_j}u^3\eta_I\, \d Y\d s\notag\\
=&G-(p-2)\iiint|\nabla_X u|^{a-2}u_{x_k}u_{x_kx_mx_ix_j}u_{x_m}u_{x_ix_j}u^3\eta_I\, \d Y\d s.
\end{align}
By partial integration,
\begin{align}\label{sq4Lkhhag}
K_{22}=&G-(p-2)\iiint|\nabla_X u|^{a-2}u_{x_k}u_{x_kx_mx_ix_j}u_{x_m}u_{x_ix_j}u^3\eta_I\, \d Y\d s\notag\\
=&G+(p-2)\iiint|\nabla_X u|^{a-2}u_{x_k}u_{x_ix_jx_k}u_{x_m}u_{x_ix_jx_m}u^3\eta_I\, \d Y\d s\notag\\
=&G+(p-2)\iiint|\nabla_X u|^{a-2}(\nabla_X u\cdot\nabla_X u_{x_ix_jx_m})^2u^3\eta_I\, \d Y\d s.
\end{align}
Hence,
\begin{align}\label{sq4Lklll}
K=&G-K_2=G-K_{22}\notag\\
=&G-(p-2)\iiint|\nabla_X u|^{a-2}(\nabla_X u\cdot\nabla_X u_{x_ix_jx_m})^2u^3\eta_I\, \d Y\d s.
\end{align}
In particular,
\begin{align}\label{sq4Lklll-}
(p-2)\iiint|\nabla_X u|^{a-2}(\nabla_X u\cdot\nabla_X u_{x_ix_jx_m})^2u^3\eta_I\, \d Y\d s+K&=G.
\end{align}
As $p>2$, then we can handle $K$, and hence one part of Corollary \ref{thm1-b} is proved.

Next we consider
\begin{align}\label{sq3ml+}
L= \iiint|\nabla_X u|^{b}(u_{tx_k})^2u^3\eta_I\, \d Y\d s.
\end{align}
We first note, using the equation, that
\begin{align}\label{sq1ml}
u_{tx_k}=&|\nabla_X u|^{p-2}\Delta u_{x_k}+(p-2)|\nabla_X u|^{p-4}(\Delta_\infty u)_{x_k}\notag\\
&+(p-2)|\nabla_X u|^{p-4}u_{x_m}u_{x_mx_k}\Delta u\notag\\
&+(p-2)(p-4)|\nabla_X u|^{p-6}u_{x_m}u_{x_mx_k}\Delta_\infty u.
\end{align}
Using this,
\begin{align}\label{sq3ml++}
L&= \iiint|\nabla_X u|^{b+p-2}(u_{tx_k})(\Delta u_{x_k})u^3\eta_I\, \d Y\d s\notag\\
&+(p-2)\iiint|\nabla_X u|^{b+p-4}(u_{tx_k})((\Delta_\infty u)_{x_k})u^3\eta_I\, \d Y\d s\notag\\
&+(p-2)\iiint|\nabla_X u|^{b+p-4}(u_{tx_k})(u_{x_m}u_{x_mx_k}\Delta u)u^3\eta_I\, \d Y\d s\notag\\
&+(p-2)(p-4)\iiint|\nabla_X u|^{b+p-6}(u_{tx_k})(u_{x_m}u_{x_mx_k}\Delta_\infty u)u^3\eta_I\, \d Y\d s\notag\\
&=: L_1+(p-2)L_2+(p-2)L_3+(p-2)(p-4)L_4.
\end{align}
Using Cauchy-Schwarz with $\epsilon$,
\begin{align}\label{sq3ddd}|L_1|+|L_2|+|L_3|+|L_4|\leq& (\epsilon_1+\epsilon_2)L+c(\epsilon_1)\iiint|\nabla_X u|^{b+2p-6}(u_{x_mx_k})^2u\eta_I\, \d Y\d s\notag\\
&+c(\epsilon_2)\iiint|\nabla_X u|^{b+2p-6}(u_{x_mx_kx_l})^2u^3\eta_I\, \d Y\d s.
\end{align}
The estimate for $L$ now follows from Theorem \ref{thm1-a}, and from the estimate of $K$ proved above. This completes the proof of Corollary \ref{thm1-b}.

\section{Parabolic uniform rectifiability, regular Lip(1,1/2) graphs, parabolic measure}\label{Upara}
As mentioned, the notion of parabolic uniformly rectifiability was introduced in \cite{HLN}, \cite{HLN1}, and concerns  time-varying boundaries which lack differentiability, and which are locally not necessarily given by graphs. For more recent developments concerning parabolic uniformly rectifiability and related topics, we refer to \cite{NS,BHHLN1,BHHLN2,BHHLN3,BHHLN4,BHHLN5}.
\subsection{Parabolic uniform rectifiability}
Let
$\Sigma \subset \mathbb R^{n+1}$ be a closed set.  Assume that $\Sigma$ is {parabolic Ahlfors-David regular} with constant $M\geq 1$. We introduce
\begin{eqnarray*} \gamma( Z, \tau, r  ):=\gamma_\Sigma( Z, \tau, r  ):= \inf_{P \in \mathcal{P}}  \biggl ( \, \bariint_{  \Delta ( Z, \tau,r) }  \, \biggl (\frac {\dist ( Y,s, P )}{r}\biggr )^2  \d \sigma (Y, s )\biggr )^{1/2}, \end{eqnarray*}
whenever $(Z,\tau)\in \Sigma $, $r>0$, and where  $\mathcal{P}$  is the set of $ n $-dimensional hyperplanes $ P $ containing a line
parallel to the $ t $ axis. We also introduce $\d \nu( Z, \tau,
r  ) :=\d \nu_\Sigma ( Z, \tau,
r  )  \, =  \, (\gamma_\Sigma  ( Z, \tau, r))^2 \, \d \sigma ( Z, \tau) \, r^{ - 1 }
\d r$.  Recall that $ \nu$  is defined  to be  a Carleson measure on $    \Delta( Y,s,R )  \times ( 0, R )$,  if there exists $ \Gamma <
\infty $ such that $ \nu ( \Delta( X, t,\rho )  \times ( 0, \rho) ) \, \leq \,
\Gamma  \, \rho^{d}$, whenever $ ( X, t  ) \in \Sigma  $ and $ C( X, t, \rho ) \subset C( Y, s, R)$.  The least such $ \Gamma  $  is
called the Carleson norm of $\nu$ on  $\Delta( Y,s,R ) \times ( 0, R ) $.
\begin{definition}\label{def1.UR}
Assume that $\Sigma  \subset \mathbb R^{n+1}$ is parabolic ADR  in the sense of Definition \ref{def1.ADR} with constant $M$. Let $\nu=\nu_\Sigma$ be defined as above. Then
 $\Sigma$ is parabolic uniformly rectifiable with constants $(M,\Gamma)$ if
\begin{eqnarray}\label{eq1.sf}
\| \nu \|:=\sup_{(X,t)\in\Sigma,\ \rho>0}  \rho^{ -d }\nu ( \Delta(X,t,\rho) \times ( 0, \rho) ) \, \leq \,
\Gamma.
\end{eqnarray}
Furthermore, if $(X_0,t_0)\in \Sigma$, $r_0>0$, then we say that $\Sigma\cap C(X_0,t_0,r_0)$ is locally parabolic uniform rectifiable with constants $(M,\Gamma)$, if
\begin{eqnarray}\label{eq1.sfll}
\sup_{(X,t)\in\Sigma,\ C(X,t,\rho)\subset C(X_0,t_0,r_0)}  \rho^{ -d }\nu ( \Delta(X,t,\rho) \times ( 0, \rho) ) \, \leq \,
\Gamma.
\end{eqnarray}
\end{definition}

\subsection{Lip(1,1/2) and regular Lip(1,1/2) graphs} Given a function $\psi:\mathbb R^{n-1}\times\mathbb R\to \mathbb R$  we let $D_{1/2}^t \psi  (x, t) $ denote
the $ 1/2 $ derivative in $ t $ of $ \psi ( x, \cdot ), x $ fixed.
This half derivative in time can be defined by way of the Fourier
transform using the multiplier $|\tau|^{1/2}$, or by
\begin{eqnarray} \label{1.8}
 D_{1/2}^t  \psi (x, t)  \equiv \hat c \int_{ \mathbb R }
\, \frac{ \psi ( x, s ) - \psi ( x, t ) }{ | s - t |^{3/2} } \, \d s,
\end{eqnarray} for properly chosen $ \hat c$. We let $ \| \cdot \|_* $ denote the
norm in parabolic $BMO(\mathbb R^{n})$ (replace standard cubes by parabolic cubes in the definition of $BMO$). A function $\psi:\mathbb R^{n-1}\times\mathbb R\to \mathbb R$ is called Lip(1,1/2) with constant $b_1$, if
\begin{eqnarray}\label{1.1}
|\psi(x,t)-\psi(y,s)|\leq b_1(|x-y|+|t-s|^{1/2})\
\end{eqnarray}
whenever $(x,t)\in\mathbb R^{n}$, $(y,s)\in\mathbb R^{n}$. If $\Sigma = \{(x, \psi(x,t), t): (x,t) \in \mathbb{R}^{n-1} \times \mathbb{R}\}$ in the coordinates $P \times P^\perp \times \mathbb{R}$, for some $t$-independent plane $P=P_x \in \mathcal{P}$ and for some Lip(1,1/2) function $\psi$, then we say that $\Sigma$ is a Lip(1,1/2) graph. An open set $\Omega\subset\mathbb R^{n+1}$ is said to be a
(unbounded) Lip(1,1/2)   graph
domain, with constant $b_1$, if
\begin{eqnarray}\label{1.1a}
\Omega=\Omega_\psi=\{(x,x_n,t)\in\mathbb R^{n-1}\times\mathbb R\times\mathbb R:x_n>\psi(x,t)\},
\end{eqnarray} for some Lip(1,1/2)  function $\psi$ having Lip(1,1/2)  constant bounded by $b_1$.

\begin{definition}\label{goodgraph.def}
 We say that $ \psi = \psi ( x, t ) : \mathbb R^{n-1}\times\mathbb R\to \mathbb R$ is a {regular} Lip(1,1/2)  function
with parameters $b_1$ and $b_2$,
if $\psi$ satisfies \eqref{1.1} and if
 \begin{eqnarray} \label{1.7}
 D_{1/2}^t\psi\in BMO(\mathbb R^n), \ \ \|D_{1/2}^t\psi\|_*\leq b_2<\infty.
\end{eqnarray}
  If $\Sigma = \{(x, \psi(x,t), t): (x,t) \in \mathbb{R}^{n-1} \times \mathbb{R}\}$ in the coordinates $P\times P^\perp \times \mathbb{R}$, for some $t$-independent plane $P=P_x\in \mathcal{P}$ and for some {regular} Lip(1,1/2) function $\psi$, then we say that $\Sigma$ is a {regular} Lip(1,1/2) graph.
\end{definition}

\begin{remark}  One can prove that in general $\psi$ being regular Lip(1,1/2) is strictly stronger than $\psi$ being Lip(1,1/2), i.e., there are examples of
functions $\psi$ which are  Lip(1,1/2) but not regular Lip(1,1/2), see \cite{LS},
\cite{KW}. One can also prove, in the context of Lip(1,1/2) graphs, that the graph being regular Lip(1,1/2) is equivalent to the graph being parabolic uniform rectifiable, see \cite{HLN1}.
\end{remark}

\begin{remark}\label{reggraph} Given $(X,t)\in \Sigma =\Sigma_\psi= \{(x, \psi(x,t), t): (x,t) \in \mathbb{R}^{n-1} \times \mathbb{R}\}$, and $r>0$, we let $\hat C(X,t,r)$ denote the orthogonal projection of
$C(X,t,r)$ onto $\mathbb{R}^{n-1} \times \mathbb{R}$. Consider $(X_0,t_0)\in \Sigma$, $r_0>0$, and assume that $\Sigma\cap C(X_0,t_0,r_0)=\Sigma_\psi\cap C(X_0,t_0,r_0)$ is locally parabolic uniform rectifiable with constants $(M,\Gamma)$ in the sense of \eqref{eq1.sfll}. Let $\phi=\phi(x,t)\in C_0^\infty(\hat C(X_0,t_0,2r_0/100))$ be such that
$0\leq\phi\leq 1$, $\phi\equiv 1$ on $\hat C(X_0,t_0,3r_0/200)$, and let $\tilde\psi(x,t):=\phi(x,t)\psi(x,t)$. Then
\begin{eqnarray} \label{1.7aga-}\sup_{(X,t)\in\Sigma_{\tilde\psi},\ r>0}\|\nu_{\tilde\psi}\|(\Sigma_{\tilde\psi}\cap C(X,t,r))\times (0,r))<\infty
\end{eqnarray}
and
 \begin{eqnarray} \label{1.7aga}
 D_{1/2}^t\tilde\psi\in BMO(\mathbb R^n), \ \ \|D_{1/2}^t\tilde\psi\|_*<\infty.
\end{eqnarray}
These conclusions are proved in Theorem 11 in \cite{LeNy}. In particular, $\Sigma_{\psi}\cap C(X,t,r_0/100)$ is given as a part of the graph of a (unbounded) {regular} Lip(1,1/2)  function.
\end{remark}

 Given a (unbounded) Lip(1,1/2)   graph domain $\Omega$  with constant $b_1$, $(X,t)=(x,x_n,t)\in\partial\Omega$, $r>0$,  we introduce reference points
\begin{align}\label{pointsref2}
A_{r}^\pm(x,t):=(x,x_n+4b_1r,t\pm r^2),\  A_{r}(x,t):=(x,x_n+4b_1r,t).
\end{align}
Furthermore, given $(X,t)=((x,\psi(x,t),t)\in\partial\Omega$,  and $\eta>0$, we introduce the (non-tangential) cone
\begin{equation}\label{nt-cone}
    \Gamma^\eta(X,t)
    := \{(Y,s)\in \Omega \mid  \dist(Y,s,X,t) < \eta |x_n-\psi(x)|\}.
\end{equation}
Given a function $u$ defined in $\Omega$, a function $f$ defined on $\partial \Omega$, and $(X,t)\in\partial \Omega $, we say that $u(X,t)=f(X,t)$ non-tangentially (n.t.) if
\[
\lim_{\substack{(Y,s)\in \Gamma^\eta(X,t)\\ (Y,s)\to (X,t)}}u(Y,s)=f(X,t),
\]
where $\eta=\eta(b_1)$ is chosen so that $\partial \Omega \cap\Gamma^\eta(X,t)=\{(X,t)\}$.  For $\eta$ fixed as stated, we let $\Gamma(X,t):=\Gamma^\eta(X,t)$.

\subsection{Rademacher's theorem for regular Lip(1,1/2)} Using that a regular Lip(1,1/2) function $\psi:\mathbb R^{n-1}\times\mathbb R\to \mathbb R$ is Lipschitz in the Euclidean sense in the spatial variables, one can apply Rademacher's theorem and conclude that there exists,
for $\d x\d t$ a.e. $(x,t) \in \mathbb R^{n - 1} \times \mathbb R$,  a linear map $M_{(x,t)} :\mathbb R^{n - 1} \to \mathbb R$ such that
\begin{equation} \frac{|\psi(y,t) - \psi(x,t) - M_{(x,t)}(y - x)|}{|x - y|} = \epsilon_{(x,t)}(|x - y|), \end{equation}
where $\epsilon_{(x,t)}(r) \to 0$ as $r \to 0$. However, the following theorem, see Theorem 3.10 in \cite{Orp}, states that this differentiability can be upgraded. In particular, the linear map determined by the horizontal gradient automatically approximates $\psi$ also in the vertical direction, $\d x\d t$ almost everywhere. This is an analogue of Rademacher's theorem for  regular Lip(1,1/2) functions.
 \begin{theorem}\label{rademacher} Let $\psi:\mathbb R^{n-1}\times\mathbb R\to \mathbb R$  be regular Lip(1,1/2). Then there exists, for  $\d x\d t$ a.e. $(x,t) \in \mathbb R^{n - 1} \times \mathbb R$, a linear map $M_{(x,t)}: \mathbb R^{n - 1} \to \mathbb R$ such that
 \begin{displaymath} \frac{|\psi(y,s) - \psi(x,t) - M_{(x,t)}(y - x)|}{\dist(x,t,y,s)} = \epsilon_{(x,t)}(\dist(x,t,y,s)), \end{displaymath}
 where $\epsilon_{(x,t)}(r) \to 0$ as $r \to 0$.
 \end{theorem}

 \begin{remark} Note that  the linear map $M_{(x,t)}: \mathbb R^{n - 1} \to \mathbb R$ in Theorem \ref{rademacher} can be identified, for  $\d x\d t$ a.e. $(x,t) \in \mathbb R^{n - 1} \times \mathbb R$, with
 $\nabla_x\psi(x,t)$ in the sense that $M_{(x,t)}(y - x)=\nabla_x\psi(x,t)\cdot(y-x)$.
 \end{remark}

\begin{definition}\label{ddeffa} We introduce, for $\d x\d t$ a.e. $(x,t) \in \mathbb R^{n - 1} \times \mathbb R$,
\begin{align*}
\mathbf{n}(X,t):=(\nabla_x\psi(x,t),1)/\sqrt{1+|\nabla_x\psi(x,t)|^2}.
\end{align*}
\end{definition}

\begin{remark}\label{remtang}
Consider a {regular} Lip(1,1/2) graph $\Sigma = \{(x, \psi(x,t), t): (x,t) \in \mathbb{R}^{n-1} \times \mathbb{R}\}$. Then Theorem \ref{rademacher} implies that there exists, for $\sigma$-a.e. $(X,t)=(x, \psi(x,t), t)\in\Sigma$, a time-independent hyperplane
    $ T (X,t)$, with unit normal $\mathbf{n}(X,t)$, such that
            \begin{align} \label{3.16tang}  \lim_{\rho \rar 0} \frac{ H ( T (X,t ) \cap C ( X,t, \rho ), \Delta(X,t,\rho))}{ \rho} =  0,
            \end{align}
      where $H$  denotes parabolic Hausdorff distance, see \eqref{distb}.
      \end{remark}

\subsection{Parabolic measure for linear parabolic operators in divergence form}\label{pmeasure} We here consider linear operators  \begin{eqnarray}\label{e-kolm-nd}
   \H=\H_A:=\partial_t-\nabla_X\cdot(A(X,t)\nabla_X),
    \end{eqnarray}
    in $\mathbb R^{n+1}$, $n\geq 1$. We assume that $A=A(X,t)=\{A_{ij}(X,t)\}_{i,j=1}^{n}$ is a real-valued, $n\times n$-dimensional, symmetric matrix
    satisfying
    \begin{eqnarray}\label{eq2}
      \kappa^{-1}|\xi|^2\leq \sum_{i,j=1}^{n}A_{ij}(X,t)\xi_i\xi_j,\quad \ \ |A(X,t)\xi\cdot\zeta|\leq \kappa|\xi||\zeta|,
    \end{eqnarray}
    for some $\kappa\in [1,\infty)$, and for all $\xi,\zeta\in \mathbb R^{n}$, $(X,t)\in\mathbb R^{n+1}$. We refer to $\kappa$ as the  constant of $A$. Assuming that $\Omega\subset\mathbb R^{n+1}$ is a (unbounded)  Lip(1,1/2)   graph domain  with constant $b_1$,  it follows, for $\varphi\in
C_0(\partial\Omega)$ given, that there exists
a  unique  weak solution  $u=u_\varphi$, $u\in C(\bar \Omega)$, to the Dirichlet problem
\begin{equation} \label{e-bvpuu}
\begin{cases}
	\H u = 0  &\text{in} \ \Omega, \\
      u = \varphi  & \text{on} \ \partial \Omega.
\end{cases}
\end{equation}
Furthermore, there exists, for every $(X,t)\in \Omega$, a unique probability
measure  $\omega(X,t,\cdot)$ on $\partial\Omega$ such that
\begin{eqnarray}  \label{1.1xxuu}
u(X,t)=\iint_{\partial\Omega}\varphi(\tilde X,\tilde t)\d \omega(X,t,\tilde X,\tilde t).
\end{eqnarray}
The measure $\omega(X,t,E)$ is referred to as the parabolic measure associated to $\H$ in $\Omega$, at $(X, t)\in \Omega$, and of $E\subset\partial\Omega$. In the case of the heat operator, we refer to $\omega(X,t,\cdot)$ as the caloric measure. Properties of $\omega(X,t,\cdot)$ govern the Dirichlet problem in \eqref{e-bvpuu}.

The Dirichlet problem, parabolic measure, and the boundary behaviour of non-negative solutions,  for linear uniformly parabolic equations with space and time-dependent coefficients, read $\H$, in Lip(1,1/2) domains have been studied intensively over the years.
Results include Carleson type estimates, the relation between the associate parabolic measure and the Green function, the backward in time Harnack inequality, the doubling of parabolic measure, boundary Harnack principles (local and global) and  H\"older continuity up to the boundary of quotients of non-negative solutions vanishing on the lateral boundary, we refer to \cite{FGSit,FGS,FS,FSY,N} for details. In particular, we refer to \cite{N} for the proof of the following theorem.

\begin{theorem}\label{dub} Assume that $\Omega\subset\mathbb R^{n+1}$ is a (unbounded)  Lip(1,1/2)   graph domain  with constant $b_1$.  Assume that  $A$ satisfies \eqref{eq2} with constant $\kappa$.  Then there exist constants $c_1=c_1(n,b_1)$ and  $c_2=c_2(n,\kappa, b_1)$,  $1\leq c_1,\ c_2<\infty$, such that
if $(X_0,t_0)\in\partial\Omega$, $r_0>0$, $(X,t)\in \partial\Omega$, $r>0$, $\Delta(X,t,r)\subset \Delta(X_0,t_0,r_0/c_1)$, then
\begin{eqnarray*}
\omega(A_{r_0}^+(X_0,t_0), \Delta(X,t,2r))\leq c_2 \omega(A_{r_0}^+(X_0,t_0), \Delta(X,t,r)).
\end{eqnarray*}
\end{theorem}

The notion of regular Lip(1,1/2) graphs, or parabolic uniform rectifiable graphs, is deeply rooted in the study of the Dirichlet problem for the heat equation in time-varying (graph) domains, and the solvability of the $L^p$-Dirichlet problem for the heat equation is intimately connected to quantitative mutual absolute continuity of  the caloric measure with respect to the surface measure. In particular, while one can prove that there are Lip(1,1/2) graph domains for which the caloric measure and the surface measure are mutually singular, in  \cite{LS,LM} it is proved that for regular Lip(1,1/2) graph domains, the caloric measure and the surface measure are quantitatively related in the sense that they are mutual absolutely continuous, and the associated parabolic Poisson kernel satisfies a scale-invariant reverse H{\"o}lder inequality in $L^p$ for some $p\in (1,\infty)$: the $A_\infty$-property of caloric measure. The importance of regular Lip(1,1/2) graph domains, from the perspective of parabolic singular integrals, layer potentials and boundary value problems, is emphasized through the works in \cite{LS,LM,H,HL}. In particular, in \cite{HL}
the solvability of the $L^2$-Dirichlet
problem (and of the $L^2$-Neumann and $L^2$-regularity problems)  for the heat equation was obtained using layer potentials in the region above a regular Lip(1,1/2) graph under the restriction that 1/2-order time derivative (measured in BMO) of the function defining the graph is small. This smallness is sharp in the
sense that there are regular Lip(1,1/2) graph domains for which the $L^2$-Dirichlet problem is not solvable.
On the other hand, the $L^p$-Dirichlet problem is solvable, for some $p<\infty$, for all
regular Lip(1,1/2) graph domains \cite{LM}.

The $L^p$ Dirichlet problem for operators with space and time dependent coefficients, connected to regular Lip(1,1/2) graph domains and also allowing for singular drift terms, was studied in the  influential work \cite{HL1}.  In \cite{HL1} the method of extrapolation of Carleson measure estimates was introduced, a method that was crucial in the resolution of the Kato conjecture, see \cite{AHLeMcT},\cite{AHLMcT}.  In this paper we will use some results concerning parabolic measure which emanates from \cite{HL1} but which are proved in \cite{NO}. Recall that $\delta(X,t)$ denote the parabolic distance from $(X,t)\in\Omega$ to $\partial\Omega$. Consider the following measures, $\mu_1$ and $\mu_2$, defined on $\Omega$,
\begin{equation}\label{measure1}
\begin{split}
\d\mu_1(X,t)&:=|\nabla_XA(X,t)|^2\delta(X,t)\ \d X\d t,\\
\d\mu_2(X,t)&:=|\partial_tA(X,t)|^2\delta^3(X,t)\ \d X\d t.
\end{split}
\end{equation}
We say that $\mu_1$ and $\mu_2$ are Carleson measures on $\Omega$ with constant $\Upsilon$ if
\begin{align}\label{measure2}
r^{-(n+1)}\iiint_{C(X,t,r)}\d\mu_1(Y,s)&\leq\Upsilon,\notag\\
r^{-(n+1)}\iiint_{C(X,t,r)}\d\mu_2(Y,s)&\leq\Upsilon,
\end{align}
for all $(X,t)\in\partial\Omega$, $r>0$. We refer to \cite{NO} for the following result.

\begin{theorem}\label{Ainfty} Assume that  $\Omega\subset\mathbb R^{n+1}$ is a (unbounded)  regular Lip(1,1/2)   graph domain with constants $(b_1,b_2)$. Assume that $A$ satisfies \eqref{eq2} with constant $\kappa$. Assume that the measures $\mu_1$ and $\mu_2$ defined in \eqref{measure1} are Carleson measures on $\Omega$ with constant $\Upsilon$ in the sense of \eqref{measure2}. Then there exist constants $c_1=c_1(n,b_1)$ and  $c_2=c_2(n,\kappa,b_1,b_2,\Upsilon)$,  $1\leq c_1,\ c_2<\infty$, and
   $\eta=\eta(n,\kappa,b_1,M_2,\Upsilon)$, $0<\eta<1$, such that the following is true. If  $(X_0,t_0)\in\partial\Omega$, $r_0>0$, $(X,t)\in \partial\Omega$, $r>0$, $\Delta(X,t,r)\subset \Delta(X_0,t_0,r_0/c_1)$, then

 \begin{eqnarray*}
 \quad c_2^{-1}\biggl (\frac{ \sigma ( E ) }{ \sigma(\Delta(X,t,r))}\biggr )^{1/\eta}\leq \frac {\omega\bigl (E\bigr )}{\omega\bigl ( \Delta(X,t,r)\bigr )}\leq c_2\biggl (\frac{ \sigma ( E ) }{ \sigma(\Delta(X,t,r))}\biggr )^\eta,
\end{eqnarray*}
whenever $E\subset \Delta(X,t,r)$.
\end{theorem}

\begin{remark}
Theorem \ref{Ainfty} is one of several equivalent formulations of  the $A_\infty$-property for parabolic measure with respect to the surface measure $\sigma$, see \cite{N,NO}.
\end{remark}

\section{Estimates for the evolutionary $p$-Laplace equation}\label{Pest}
In this section we focus on estimates for the evolutionary $p$-Laplace equation and we derive estimates based on the assumptions that $|\nabla_Xu|> 0$.
The Schauder type results we derive are well known to the experts, but we believe that the proofs and the brief introduction to the regularity theory for the
evolutionary $p$-Laplace equation can serve  the interested reader. Due to the lack of homogeneity of the evolutionary $p$-Laplace equation when $p\neq 2$, a crucial ingredient in the regularity theory for this equation is the use of DiBenedetto's intrinsic geometry when deriving local estimates. This amounts to using cylinders whose size depends on the solution itself.

\subsection{Notation} Given $\O\subset\mathbb R^{n+1}$ and a function $f:\O \to\er^m$, $m\geq 1$, we let
$$\osc_{\O} f:=\sup\limits_{(X_0,t_0),(X,t)\in \O}|f(X_0,t_0)-f(X,t)|$$
denote the oscillation of $f$ on $\O$. Given $(X,t)\in \mathbb R^{n+1}$ and $r,\lambda>0$, we introduce the space-time cylinders
\begin{align}\label{eq:gen int cyl +}
 Q_r^{\lambda}(X,t)&:=B(X, r) \times (t- \lambda^{2-p} r^p,t +  \lambda^{2-p} r^p).
\end{align}
In a context where the dependence on $(X,t)$ is not important we will often write $Q_{r}^{\lambda}$ for $ Q_{r}^{\lambda}(X,t)$. Furthermore, we let $\omega:\er_+ \mapsto \er_+$ be a concave modulus of continuity, i.e.,   a concave nondecreasing function
such that $\omega(1) = 1$ and $\omega(0):=\lim_{r \downarrow 0}\omega(r) = 0$. Concavity of $\omega(\cdot)$ implies that
\begin{equation}\label{eq:omega 1}
 \frac{\omega(r)}{r}\leq \frac{\omega(\varrho)}{\varrho}\quad \mbox{ whenever $0<\varrho < r$}.
\end{equation}
Furthermore, \begin{equation}\label{eq:int cyl rel 2}
    \left(\varrho/r \right)^{(p-2)/p} Q_{\varrho}^{\lambda \omega(\varrho)} \subset Q_{\varrho}^{\lambda \omega(r)} \qquad  \mbox{for every} \  \varrho \in (0,r)\,.
\end{equation}
A generic example is $\omega(r)=\omega_\alpha(r)=r^\alpha$, $\alpha\in (0,1]$. Following \cite{KMN1}, given a function $f$ defined on $\O = U\times (t_1, t_2) \subset \er^n \times \er$, we set
\eqn{funcspac}
$$
\left[f\right]_{\widetilde C^{\omega(\cdot)}(\O)}:= \inf \left\{ \lambda > 0 \ : \ \sup\limits_{Q_{r}^{\lambda \omega(r)} \subset \er^n \times \er} \left(\frac1{\lambda \omega(r)}\osc_{Q_{r}^{\lambda \omega(r)} \cap \, \O} f \right) \leq 1 \right\}\,.
$$
Moreover, we let $C^0(\O)$ denote the set of functions which are continuous on $\O$. We note that in the special case $\omega(r)=r^\alpha$, $\alpha \in (0,1]$, then the definition in \eqref{funcspac} reduces to a notion of  H\"older continuity,
\begin{equation*}\label{eq:holder}
    \omega(r)=r^\alpha\,, \quad \alpha \in (0,1]\,, \quad \left[f\right]_{\widetilde C^{\omega(\cdot)}(\O)} < \infty \quad \Longleftrightarrow \quad
    \sup_{(X_1,t_1), (X_2,t_2) \in\O}\frac{|f(X_1,t_1)- f(X_2,t_2)|}{\|(X_1,t_1)-(X_2,t_2)\|_{\alpha}^\alpha} < \infty\,,
\end{equation*}
where
\begin{align}\label{intrim}
\|(X_1,t_1)-(X_2,t_2)\|_{\alpha} := \max\left\{|X_1-X_2|,|t_1-t_2|^{1/[p-\alpha(p-2)]}\right\} \,.
\end{align}
In particular, this metric is depending on the degree of regularity considered. Note also that when $p=2$, then the space introduced coincides with the space of functions which are  H\"older continuous of order $\alpha$  with respect to the standard parabolic metric.

\subsection{Intrinsic geometry: the philosophy} The lack of homogeneity of the evolutionary $p$-Laplace equation results in the lack of {homogeneous} a priori estimates, and the impossibility of using such estimates in iterative schemes in line with the standard regularity techniques. Instead, the lack of homogeneity must be locally corrected by using intrinsic geometries, and the philosophy is that the type of cylinders used must depend on the type of regularity one is proving.  To give an illustrative example following \cite{KMN1}, let us discuss the type of geometry used in the case one is interested in proving gradient regularity starting from higher integrability of the gradient, see  \cite{DBF,DB,KL}. In this case the relevant cylinders are
\begin{align}\label{geo1}
Q_{r}^{\lambda\omega(r)}(X_0,t_0):= B(X_0, r) \times (t_0- \lambda^{2-p}(\omega(r))^{2-p} r^p,t_0+\lambda^{2-p}(\omega(r))^{2-p} r^p),
\end{align}
where
\begin{align}\label{geo1+}
\left(\bariiint_{Q_{r}^{\lambda \omega(r)}(X_0,t_0)} |\nabla_Xu|^{p}\, \d X \, \d t \right)^{1/p} \approx \lambda \mbox{ and }\omega(r)=r\,.
\end{align}
The first relation encodes the fact that on $Q_{r}^{\lambda r}(X_0,t_0)$ we have, in an integral sense, $|\nabla_Xu|\approx \lambda $. We now proceed heuristic as follows. On $Q_{r}^{\lambda r}(X_0,t_0)$ we formally identify
$$
\partial_tu - \nabla_X\cdot(|\nabla_Xu|^{p-2}\nabla_Xu) \approx u_t - \lambda^{p-2}\Delta_X u\,.
$$
Therefore, with this heuristics $v(X,t):= u(X_0+rX, t_0+ \lambda^{2-p}r^2t)$ solves the heat equation $\partial_t v -\Delta_X v=0$ in $B(0,1) \times (-1,1)$ and {homogenous} estimates can be derived which are suitable for regularity procedures. To make this rough argument rigorous is far from being trivial, but the point is that for  this procedure to work, along the iteration,  {the gradient must remain bounded}. In other words, {the type of intrinsic geometry considered depends on the kind of regularity one is proving.} For the same reason,  when proving regularity results for $u$, see for instance \cite{DB, DGV1-, DGV1, K}, one is led to use the geometry dictated by the cylinders in \eqref{geo1} but now with $\omega(r)$ and $\lambda$ satisfying
\begin{align}\label{geo2}
\osc_{Q_{r}^{\lambda\omega(r)}(X_0,t_0)}\, u\approx \lambda \omega(r)\mbox{ and }\omega(r)\equiv 1\,.
 \end{align}
  An observation is that the two geometries considered in \eqref{geo1+} and \eqref{geo2} are two particular, actually extremal, cases of a class of intermediate/interpolative intrinsic geometries, suited to the regularity one want to prove. In particular, if we let $\omega(r)=r^\alpha$, $\alpha\in (0,1$], then the geometries in  \eqref{geo1+} and \eqref{geo2} can be seen to contain the endpoint cases $\omega(r)=r$ and $\omega(r)\equiv 1$, respectively.  The limiting cases of the parabolic metric used in \eqref{intrim}
are, in the case $\alpha=1$,
$$
\|(X_1,t_1)-(X_2,t_2)\|_{1} = \max\left\{|X_1-X_2|,|t_1-t_2|^{1/2}\right\},
$$
and this the usual parabolic metric used to study the regularity of the gradient, and, when $\alpha \to 0$,
$$
\|(X_1,t_1)-(X_2,t_2)\|_{0} = \max\left\{|X_1-X_2|,|t_1-t_2|^{1/p}\right\},
$$
which is instead the metric that turns out to be relevant in the study of H\"older continuity of solutions, again see \cite{DB}. For the efficiency
of these intermediate/interpolative intrinsic geometries in the study of optimal regularity in the $p$-parabolic obstacle problem we refer to \cite{KMN1}.

\subsection{Energy and zero order estimates}

For the record we state the following Harnack estimate which can be found in \cite{DB}. For generalizations to operators of $p$-parabolic type but with only bounded and measurable coefficients we refer to ~\cite{DGV1} and ~\cite{K}.
\begin{theorem} \label{thm:harnack}
Suppose that $u$ is a nonnegative weak solution to \eqref{basic eq} in a space-time cylinder $\O$.
There  are constants $c_i=c_i(n,p)$, $i \in \{1,2\}$, such that
if
\[
B(X_0,2r) \times (t_0-c_1 u(X_0,t_0)^{2-p} r^p,t_0+c_1 u(X_0,t_0)^{2-p} r^p) \Subset \O\,,
\]
then
\begin{eqnarray*}
   && u(X_0,t_0) \leq c_2 \inf_{X\in B(X_0,r)} u\left(X,t_0+c_1 u(X_0,t_0)^{2-p} r^p\right)\,.
\end{eqnarray*}
\end{theorem}
The next result is a standard energy estimate applied in $Q_{r}^{\lambda \omega(r)}(X_0,t_0)$ (see \cite[Proposition 3.1, Chapter 2]{DB}), together with an $L^\infty$ bound for the solution which can be inferred from \cite[Theorem 4.1, Chapter 5]{DB}, with some small variants.
\begin{lemma} \label{lemma:energy}
Suppose that $w$ is a nonnegative weak subsolution to \eqref{basic eq} in $Q_r \equiv Q_{r}^{\lambda \omega(r)}(X_0,t_0)$.
Then there exists a constant $c=c(n,p)$ such that
\begin{eqnarray}
\nonumber
   && \iiint_{Q_{r/2}} |\nabla_Xw|^p \, \d X \d t  + \sup_{t_0-(\lambda \omega(r/2))^{2-p}(r/2)^p<t< t_0} \iint_{B(X_0,r/2)} w^2(\cdot, t)\, \d X   \\
\label{eq:energy}   &&  \qquad \leq \frac{c}{r^p} \iiint_{Q_{r}}\left[ w^p + (\omega(r)\lambda)^{p-2} w^2\right] \, \d X \d t
\end{eqnarray}
holds. Furthermore, let $\eps>0$ be a degree of freedom. Then there exists a constant $c_\eps\geq 1$, depending only on $n,p,\eps$, such that
\begin{eqnarray}
\label{eq:sup}   && \sup_{Q_{r/2}} w  \leq \eps\omega(r)\lambda  + c_\eps(\omega(r)\lambda)^{2-p} \bariiint_{Q_{r}} w^{p-1} \d X \d t\,.
\end{eqnarray}
\end{lemma}

\subsection{Gradient estimates}
The first auxiliary theorem stated below gives an estimate of the local supremum of the gradient in the form of a reverse H\"older inequality. In the case of
the equation in \eqref{basic eq}, the estimate can be found in  \cite[Chapter 8, Theorem 5.1]{DB}. The second estimate below is a consequence of the first estimate, a  simple covering argument and~\eqref{eq:int cyl rel 2}.
\begin{theorem} \label{thm:grad bound}
Suppose that $u$ is a weak solution to \eqref{basic eq} in $Q_{r}^{\lambda r}$ for some $r,\lambda>0$ and let $\eps >0$ be a degree of freedom. Then there exists a constant $c_\eps \geq 1$, depending only on $n,p,\eps$, such that
$$
\sup_{Q_{r/2}^{\lambda r}} |\nabla_Xu| \leq \eps \lambda + c_\eps \lambda^{2-p} \bariiint_{ Q_{r}^{\lambda r}} |\nabla_Xu|^{p-1} \, \d X \d t
$$
holds. In particular, if $\lambda = \widetilde \lambda \omega(r)/r$, for some $r,\tilde\lambda>0$, then
\begin{equation}\label{eq:grad bound 2}
\sup_{ Q_{r/2}^{\widetilde \lambda \omega(r/2)}} |\nabla_Xu| \leq \eps \widetilde \lambda \omega(r)/r + c_\eps (\widetilde \lambda \omega(r)/r)^{2-p} \bariiint_{ Q_{r}^{\widetilde \lambda \omega(r)}} |\nabla_Xu|^{p-1} \, \d X \d t\,.
\end{equation}
\end{theorem}

 The next fundamental regularity result was obtained for evolutionary parabolic equations in \cite{DBF}. We refer to \cite[Theorem 3.2]{KMw} and \cite[Theorem 3.2]{KMpisa} for the scalar case and for more details.
\begin{theorem}\label{osci} Suppose that $u$ is a weak solution to \eqref{basic eq} in a space-time cylinder $\O$. Then
$\nabla_Xu$ has the H\"older continuous representative in $\O$. Moreover, let $Q_{r}^{\lambda r} \subset \O$, for some $r,\lambda>0$ such that
$$
  \sup_{Q_{r}^{\lambda r}} |\nabla_Xu|\leq  M\lambda,
$$
holds for a constant $M \geq 1$. Then there exists $\alpha \equiv \alpha (n,p,M) \in (0,1]$ such that
\eqn{oscit}
$$
\osc_{Q_{\varrho}^{\lambda \varrho}}\nabla_Xu \leq 4 M\lambda \left(\frac{\varrho}{r} \right)^{\alpha},
$$
holds for all $\varrho \in (0,r)$. Here $Q^{\lambda \varrho}_{\varrho} \subset Q_{r}^{\lambda r}$, for $0<\rho\leq r$,  is an intrinsic cylinder sharing its center with $Q_{r}^{\lambda r}$.
\end{theorem}

In the intrinsic geometry suited for the general modulus of continuity, the above H\"older estimates takes the following form.

\begin{corollary} \label{cor:osci}
Let $u$ be as in Theorem~\ref{osci} with $\lambda = \widetilde \lambda \omega(r)/r$, for some $r,\widetilde \lambda>0$. Then
 $$  \osc_{ Q_{\varrho}^{\widetilde \lambda \omega(\varrho)}} \nabla_Xu \leq 4M \widetilde \lambda \frac{\omega(r)}{r} \left(\frac{\varrho}{r} \right)^{\alpha}
$$
holds for all $\varrho \in (0,r)$ and with $\alpha$ as in Theorem~\ref{osci}.
\end{corollary}

\subsection{Estimates assuming that $|\nabla_Xu|> 0$} \label{nondd} We here use (some of) the estimates of the previous subsection to deduce some elementary Schauder type estimates/conclusions assuming that $|\nabla_Xu|> 0$, i.e., assuming that the equation is locally non-degenerate.

 \begin{lemma}\label{lem:c2alpha}
Suppose that $u$ is a weak solution to \eqref{basic eq} in a space-time cylinder $\O=U\times (t_1,t_2)$. Assume that $|\nabla_Xu|> 0$ in $\O$. Then $u$ is infinitely differentiable (with respect to $X$ and $t$) in
$\O$ and $u$ is a strong solution to evolutionary $p$-Laplace equation in $\O$.
\end{lemma}
\begin{proof} In the following we can without loss of generality assume that $U$ is bounded. Let $(X_0,t_0)\in\O$ and assume that $|\nabla_Xu(X_0,t_0)|=:\delta>0$. As $u$ is a weak solution to \eqref{basic eq} in  $\O$, we have that $u\in L^p(t_1,t_2,W^{1,p}(U))$. Assume that $Q_{r}^{\lambda r}(X_0,t_0)$ is compactly contained in $\O$ for some $r,\lambda>0$. Then, using Lemma \ref{thm:grad bound}, and H{\"o}lder's inequality,
\begin{align}\label{mad1}
\sup_{Q_{r/2}^{\lambda r}(X_0,t_0)} |\nabla_Xu| &\leq \eps \lambda + c_\eps \lambda^{2-p} \bariiint_{ Q_{r}^{\lambda r}(X_0,t_0)} |\nabla_Xu|^{p-1} \, \d X \d t\notag\\
&\leq \eps \lambda + c_\eps \lambda^{2-p} \biggl (\bariiint_{ Q_{r}^{\lambda r}(X_0,t_0)} |\nabla_Xu|^{p} \, \d X \d t\biggr )^{(p-1)/p}\notag\\
&\leq \eps \lambda + c_\eps \lambda^{(2-p)/p} r^{-(n+2)(p-1)/p}\|u\|_{L^p(t_1,t_2,W^{1,p}(U))}.
\end{align}
Let
$$M:=\eps  + c_\eps \lambda^{(2-2p)/p} r^{-(n+2)(p-1)/p}\|u\|_{L^p(t_1,t_2,W^{1,p}(U))}<\infty.$$
Then \eqref{mad1} gives
\begin{align}\label{mad2}
\sup_{Q_{r/2}^{\lambda r}(X_0,t_0)} |\nabla_Xu| &\leq M\lambda.
\end{align}
Using Theorem \ref{osci} this implies that there exists $\alpha=\alpha (n,p,M) \in (0,1]$ such that
\begin{align}\label{mad3}
\osc_{Q_{\varrho}^{\lambda \varrho}(X_0,t_0)}\nabla_Xu \leq 4 M\lambda \left(\frac{\varrho}{r} \right)^{\alpha},
\end{align}
holds for all $\varrho \in (0,r)$. In particular,
\begin{align}\label{mad4}
|\nabla_Xu(X,t)-\nabla_Xu(X_0,t_0)|\leq 4 M\lambda \left(\frac{\varrho}{r} \right)^{\alpha},
\end{align}
whenever $(X,t)\in Q_{\varrho}^{\lambda \varrho}(X_0,t_0)$.  With $M$, $\lambda$, $r$ and $\alpha$ fixed, we now choose $\varrho_0$ small enough so that
\begin{align}\label{mad4+}
|\nabla_Xu(X,t)-\nabla_Xu(X_0,t_0)|\leq \delta/2.
\end{align}
Hence
\begin{align}\label{mad6}
\delta/2\leq|\nabla_Xu(X,t)|\leq 3\delta/2,
\end{align}
whenever $(X,t)\in Q_{\varrho_0}^{\lambda \varrho_0}(X_0,t_0)$. Next, formally carrying out the differentiations in the $p$-parabolic equation, we arrive at the  non-divergence form equation
\begin{align}\label{bla}
 \partial_tu-|\nabla_Xu(X,t)|^{p-2}u_{x_ix_j}- (p-2)|\nabla_Xu(X,t)|^{p-4}u_{x_ix_k}u_{x_j}=:\partial_tu-B_{ij}(X,t)u_{x_ix_j}.
\end{align}
Here,  $B_{ij} = B_{ji}$, and
\begin{equation}
	\label{eqatilde}
	c^{-1} \lambda(X,t) |\xi|^2 \leq \sum\limits_{i,j=1}^n B_{i,j} (X,t)\xi_i \xi_j \leq c \lambda(X,t) |\xi|^2,\mbox{ whenever } \xi\in \mathbb R^n,
\end{equation}
$\lambda(X,t):= |\nabla_X u(X,t)|^{p-2}$. In particular, in this form the $p$-parabolic equation is, in $Q_{\varrho_0/2}^{\lambda \varrho_0/2}(X_0,t_0)$ and as a consequence of \eqref{mad6}, a uniformly elliptic parabolic equation in non-divergence form with symmetric coefficients. Now, using Schauder theory for these equation, and uniqueness for the Dirichlet problem, see Lemma 12.11 in \cite{Lieb} or \cite{W1,W2}, it follows that $u$ is twice continuously differentiable in $X$ and once continuously differentiable in $X$. Furthermore, by a bootstrap argument, iteratively differentiating the $p$-parabolic equation with respect to $X$ and $t$, we can conclude that $u$ is infinitely differentiable (with respect to $X$ and $t$) in
$\O$, and that $u$ is a strong solution to evolutionary $p$-Laplace equation in $Q_{\varrho_0/2}^{\lambda \varrho_0/2}(X_0,t_0)$. This completes the proof of the lemma.
\end{proof}

\begin{lemma}\label{thm1-lem}  Let $p$, $2<p<\infty$, be fixed. Let $\Sigma$ be a closed subset of $\mathbb R^{n+1}$ which is parabolic Ahlfors-David regular with constant $M$, let  $\Omega:= \mathbb R^{n+1}\setminus \Sigma$. Let
$(X_0,t_0)\in \Sigma$, $r_0\in (0,\diam (\Sigma)/2)$. Assume  that $u$ is a non-negative function in $\Omega\cap C(X_0,t_0,2r_0)$  which satisfies $\partial_tu-\nabla_X\cdot(|\nabla_Xu|^{p-2}\nabla_Xu)=0$ in  $\Omega\cap C(X_0,t_0,2r_0)$. Assume in addition that there  a  constant $\gamma$, $1\leq\gamma<\infty$,  such that
\begin{align}\label{boundsaapalem} \gamma^{-1}\leq |\nabla_X u(Y,s)|,\ u(Y,s)\leq \gamma\delta(Y,s),\end{align}
for all $(Y,s)\in \Omega\cap C(X_0,t_0,r_0)$.  Then there exist  constants $c=c(n,M)\in (1,\infty)$, and $\tilde\gamma$, which only depends on $n$, $p$, $\gamma$, such that
\begin{align}\label{boundsaapalem+}\delta^2(Y,s)|\nabla_X^3u(Y,s)|+ \delta(Y,s)|\nabla_X^2u(Y,s)|+|\nabla_X u(Y,s)| \leq \tilde\gamma,\end{align}
for all $(Y,s)\in \Omega\cap C(X_0,t_0,r_0/c)$.
\end{lemma}
\begin{proof} Let $I$ be a Whitney cube of size $r$ such that $I\cap C(X_0,t_0,r_0/c)\neq 0$. Combining Theorem \ref{thm:grad bound} and Lemma \ref{lemma:energy}, with $\lambda=1$, and the assumption that $ u(Y,s)\leq \gamma\delta(Y,s)$, we immediately deduce that $|\nabla_X u(Y,s)|\leq \tilde\gamma$ on $I$, and hence
\begin{align}\label{boundsaapalem+ml} \gamma^{-1}\leq |\nabla_X u(Y,s)|\leq \tilde\gamma,\end{align}
for all $(Y,s)\in \Omega\cap C(X_0,t_0,r_0/c)$. Using Lemma \ref{lem:c2alpha} we have that $u$ is infinitely differentiable (with respect to $X$ and $t$) in
$\Omega\cap C(X_0,t_0,r_0/c)$. Now, again formally carrying out the differentiations in the $p$-parabolic equation, we arrive at the  non-divergence form equation in \eqref{bla}
with coefficients as in \eqref{eqatilde}.  The stated estimate for $\delta(Y,s)|\nabla_X^2u(Y,s)|$ and $\delta^2(Y,s)|\nabla_X^3u(Y,s)|$ now follows from \eqref{boundsaapalem+ml} and Schauder estimates, again see \cite{Lieb} or \cite{W1,W2}.
\end{proof}

\section{Applications: proof of Theorem \ref{Free} and Theorem \ref{Free+}}\label{App}

In this section we prove  Theorem \ref{Free} and Theorem \ref{Free+}.

\subsection{Proof of Theorem \ref{Free}} Let $(X_0,t_0)\in \Sigma$, $r_0>0$, and by assumptions $u$ is a smooth non-negative function in $\Omega\cap C(X_0,t_0,2r_0)$ subject to the stated restrictions. We need to prove that there exist a constant $c=c(n,b_1)\in (1,\infty)$ such that,
\begin{align}\label{conc1rep}
&\mbox{$\Sigma\cap C(X_0,t_0,r_0/c)$ is locally parabolic uniform rectifiable}\notag\\
&\mbox{with the stated control on the constants}.
\end{align}
It then follows, see Remark \ref{reggraph}, that $\Sigma\cap C(X_0,t_0,r_0/c)$ is given as a part of the graph of a (unbounded) {regular} Lip(1,1/2)  function,  with constant $b_2=b_2(n,p,b_1,\gamma)$. To prove \eqref{conc1rep} we need to prove that if $(X,t)\in\partial\Omega$ and $R>0$ satisfy $C(X,t,R)\subset C(X_0,t_0,r_0/c)$, then
\begin{eqnarray*}\nu ( \Delta(X,t,R)\times ( 0, R) ) \, \lesssim R^{ n + 1 },
\end{eqnarray*} where the implicit constant is independent of
$( X,t)$ and $R$, and only depend on $n,p,b_1$, and $\gamma$. Recall that
$$\d \nu ( Z, \tau,
r  )  \, =  \, \ga^2( Z, \tau, r) \, \d \si ( Z, \tau) \, r^{ - 1 }
\d r,$$
and
$$ \ga ( Z, \tau, r  ) = \inf_{P \in \mathcal{P}}  \biggl ( \, \bariint_{  \Delta ( Z, \tau,r) }  \, \biggl (\frac {\dist ( Y,s, P )}{r}\biggr )^2  \d \sigma (Y, s )\biggr )^{1/2}.$$
In this last expression the infimum is taken
over all  $ n $ dimensional planes $ P $ containing a line
parallel to the $ t $ axis.   In the following we consider points
$ ( \tilde X, \tilde t ) \in \Delta(X_0,t_0,r_0/c)$
and $ r<r_0$. We introduce $\mathcal{S}=\{A_r(\tilde X,\tilde t)+(0,0,\rho);\ \ \rho\in (-r^2,r^2)\}$, i.e., $\mathcal{S}$ is the line in the $t$-direction connecting
the points $A_r^+(\tilde X,\tilde t)$ and $A_r^-(\tilde X,\tilde t)$. By construction this line
is contained in $\Omega$ and for all points $(\tilde Z,\tilde\tau)\in\mathcal{S}$, $\delta(\tilde Z,\tilde\tau)\sim r$.
Let $(\hat X,\tilde t)=A_r(\tilde X,\tilde t)$ and let
$I$ denote the interval $(\tilde t-r^2,\tilde t+r^2)$.  Let $\eta$ be a small positive number. Then, by the mean value theorem
\begin{eqnarray*}
   u(\hat X,\tilde t)-u(x,x_n+\eta r,\tilde t)=
   (4b_1-\eta)r\langle\nabla_X u(x,x_n+\beta r,\tilde t),e_n\rangle,
   \end{eqnarray*}
   for some $\beta\in (\eta,4b_1)$. In particular,
   \begin{eqnarray*}
   (4\delta b_1-\gamma\eta)\leq (4b_1-\eta)\langle\nabla_X u(x,x_n+\beta r,\tilde t),e_n\rangle,
   \end{eqnarray*}
   and we can conclude, by choosing $\eta$ small, that there is a point $(X^\ast,\tilde t)\in\Omega$, $X^\ast:=(x,x_n+\beta r)$, $\beta\in (\eta,4b_1)$, such that
   \begin{eqnarray}\label{impo}
   \tilde c^{-1}\leq \langle\nabla_X u(X^\ast,\tilde t),e_n\rangle,
   \end{eqnarray}
   for some $\tilde c=\tilde c(\delta, b_1,\gamma)\geq 1$.

   Considering $ ( Y, s )\in \Omega$ with $s\in I$ and using
Taylor's formula we get
 \begin{eqnarray*}
  u ( Y, s ) =  u ( X^\ast,s ) +   \lan \nabla_X u ( X^\ast, s ), Y - X^\ast \ran
 +
\int_{X^\ast }^{Y}   \,  \lan \nabla_X^2 u ( Z, s  ), Z - Y \ran   \d l. \end{eqnarray*}
Here the last integral is interpreted as the second directional
derivative of $ u $ taken along the line segment $ l $ from $ ( X^\ast,
s)  $ to $ ( Y, s)  $ (with $ Z $ on $l$). We introduce
\begin{eqnarray*} A (X^\ast, Y, s )  :=
 u ( X^\ast, s ) +   \lan \nabla_X u ( X^\ast, s ), Y -  X^\ast \ran
 \mbox{ when }  Y \in \mathbb R^{n},\end{eqnarray*}
 and we let $P$ denote the hyperplane defined through
 \begin{align}
 P:=\{ ( Z, \tau ) : A ( X^\ast, Z, \tilde t ) = 0 \}= \{ ( Z, \tau ) : \lan \nabla_X u (  X^\ast, \tilde t ),  X^\ast -Z\ran= u (  X^\ast, \tilde t )\}.
 \end{align}
 We claim that
 \begin{align}\label{appr}
 \mbox{$|A(X^\ast, Y, \tilde t )|\sim\dist(Y,s,P)$
 for all $(Y,s)\in\Delta(\tilde X,\tilde t,r)$}.
 \end{align}
 Indeed, assume that $(Y,s)\in\Delta(\tilde X,\tilde t,r)$ and $( Z, \tau )\in P$. Then
 \begin{align}\label{appr+}
 |A (X^\ast, Y, \tilde t )|&=  |A (X^\ast, Y, \tilde t )- A (X^\ast, Z, \tilde t )|= |\lan \nabla_X u ( X^\ast, \tilde t), Y - Z \ran|\sim |Y-Z|,
 \end{align}
 by \eqref{impo} and the fact that $|\nabla_X u ( X^\ast, \tilde t)|\lesssim 1$.

  By definition
\begin{eqnarray*}
\nu ( \De ( X, t, R ) \times ( 0,
R ) )
 \leq
  {\ds \int_0^R \iint_{\De ( X, t, R)} \xi ( \tilde X, \tilde t, r )
 \d \si ( \tilde X , \tilde t)  \,  \d r },
\end{eqnarray*}
where
\begin{eqnarray*}
\xi ( \tilde X,\tilde t, r ) \, = \, r^{ - n - 4 } \,
{ \ds \iint_{ \De ( \tilde X,\tilde t, r) }  (\dist(Y,s,P))^2 \d \si ( Y, s ) }.
\end{eqnarray*}
Using the deduction above,
\begin{align*}
\xi ( \tilde X,\tilde t, r ) \, &= \, r^{ - n - 4 } \,
{ \ds \iint_{ \De ( \tilde X,\tilde t, r) }  (\dist(Y,s,P))^2 \d \si ( Y, s ) }\\
&\sim  r^{ - n - 4 }
{ \ds \iint_{ \De ( \tilde X,\tilde t, r) }  |A(X^\ast, Y, \tilde t )|^2  \d \si ( Y, s ) }\\
&\leq2 r^{ - n - 4 }
{ \ds \iint_{ \De ( \tilde X,\tilde t, r) }  |A(X^\ast, Y, s )|^2  \d \si ( Y, s ) }\\
&+2 r^{ - n - 4 }
{ \ds \iint_{ \De ( \tilde X,\tilde t, r) }  |A(X^\ast, Y, s )-A(X^\ast, Y, \tilde t )|^2  \d \si ( Y, s ) }.
\end{align*}
Hence we have to estimate $|A(X^\ast, Y, s )|^2 $ and $|A(X^\ast, Y, s )-A(X^\ast, Y, \tilde t )|^2$ for all $(Y,s)\in \De ( \tilde X,\tilde t, r)$.

 Next, given
$(Y,s)\in\Delta(\tilde X,\tilde t,r)$ we note that due to \eqref{impo}, the line
emanating at $(Y,s)$ and extending in the direction of $e_n$ will hit the plane $P$ at one unique point.
 Using the fact that $ u  = 0 $ on $  \partial \Omega $, Taylor's formula as above and Schwarz's inequality
 we get for all $s\in I$,
\begin{eqnarray}\label{appr++}
  | A ( X^\ast, Y,  s ) |^2
 \, \lesssim r \,  \int_{X^\ast }^{Y}   \, \de ^2( Z, s  ) |\nabla_X^2 u ( Z, s  )|^2 \, \d l.  \end{eqnarray}
Hence,
\begin{align*}
2 r^{ - n - 4 }
{ \ds \iint_{ \De ( \tilde X,\tilde t, r) }  |A(X^\ast, Y, s )|^2  \d \si ( Y, s ) }\lesssim r^{ - n - 3 }\iiint_{ \Omega\cap C( \tilde X,\tilde t, 100b_1r) }
|\nabla_X^2 u ( Z, s  )|^2\de^2( Z, s  )\, \d Z\d s.
\end{align*}
Also,
\begin{align}
  A ( X^\ast, Y, s ) - A ( X^\ast, Y, \tilde t) &= u ( X^\ast, s )-u ( X^\ast, \tilde t ) \notag\\
  &+   \lan (\nabla_X u ( X^\ast, s )-\nabla_X u ( X^\ast, \tilde t )), Y -  X^\ast \ran\notag\\
  &= \bigl (\partial_t u ( X^\ast, t^\ast )+\lan \nabla_X \partial_t u ( X^\ast, t^\ast ), Y -  X^\ast \ran\bigr )(s-\tilde t),
  \end{align}
  for some $t^\ast\in (\tilde t,s)$. Using interior
estimates, which again follows from our assumptions and \cite{Lieb},  we get, for $ ( Y, s )\in \De ( \tilde X,\tilde t, r) $,
\begin{align*}
  &| A ( X^\ast, Y, s ) - A ( X^\ast, Y, \tilde t) |^2\\
  &\lesssim    r^{  - n}  \iiint_{ \Omega\cap C( X^\ast, \tilde t,2r) } \,
 \bigl ( |\partial_tu(Z,s)|^2\de^2( Z, s )+|\nabla_X \partial_tu ( Z, s  )|^2\de^4( Z, s  )\bigr )\, \d Z\d s.\end{align*}
 Hence,
 \begin{align*}
&2 r^{ - n - 4 }
{ \ds \iint_{ \De ( \tilde X,\tilde t, r) }  |A(X^\ast, Y, s )-A(X^\ast, Y, \tilde t )|^2  \d \si ( Y, s ) }\\
& \lesssim r^{ - n - 3 }\iiint_{ \Omega\cap C( \tilde X,\tilde t, 100b_1r) }
|\partial_tu ( Z, s  )|^2\de^2( Z, s  )\, \d Z\d s\notag\\
&+r^{ - n - 3 }\iiint_{ \Omega\cap C( \tilde X,\tilde t, 100b_1r) }
|\nabla_X\partial_t u ( Z, s  )|^2\de^4( Z, s  )\, \d Z\d s.
\end{align*}
Using this, and continuing our previous deductions, we find
\begin{align*}
  \xi ( \tilde X, \tilde t, r )
  &\lesssim  r^{ - n - 3 } \iiint_{ \Omega\cap C( \tilde X,\tilde t, 100b_1r) }\bigl ( |\partial_tu ( Z, s  )|^2\de^2( Z, s  )+
|\nabla_X\partial_t u ( Z, s  )|^2\de^4( Z, s  )\bigr )\, \d Z\d s\\
&+  r^{ - n - 3 } \iiint_{ \Omega\cap C( \tilde X,\tilde t, 100b_1r) }|\nabla_X^2 u ( Z, s  )|^2\de^2( Z, s  )\, \d Z\d s.
\end{align*}
Integrating  this  with respect to  $ ( \tilde X, \tilde t ) \in
\Omega\cap C(X,t,R) $ and $ r \in ( 0, R ) $ we obtain, after
interchanging the order of integration,
\begin{align*}
 \nu ( \De (X,t, R ) \times ( 0,
R ) )&\leq
  {\ds \int_0^R \iint_{\De ( X,t,R)} \xi ( \tilde X, \tilde t, r )
 \d \si ( \tilde X , \tilde t)  \,  \d r }\\
 &\lesssim   \iiint_{ \Omega\cap C( X,t, 200 b_1R) }\bigl ( |\partial_tu ( Z, s  )|^2\de( Z, s  )+
|\nabla_X\partial_t u ( Z, s  )|^2\de^3( Z, s  )\bigr )\, \d Z\d s\\
&+  \iiint_{ \Omega\cap C( X,t, 200 b_1R) }|\nabla_X^2 u ( Z, s  )|^2\de( Z, s  )\, \d Z\d s.
 \end{align*}
 Using the assumptions on $u$ it follows that $\de( Z, s  )\lesssim u(Z,s)$ for all $(Z,s)\in \Omega\cap C( X,t, 200 b_1R)$.  Hence, combining Lemma \ref{thm1-lem}, Theorem \ref{thm1-a} and Corollary \ref{thm1-b},
 \begin{align*}
 \nu ( \De (X,t, R ) \times ( 0,
R ) )&\lesssim  R^{n+1},
 \end{align*}
   and this completes the proof.

\subsection{Proof of Theorem \ref{Free+}}  To prove the first statement of the theorem, \eqref{conc2}, note that for each $k\in \{1,...,n\}$, $u_{x_k}$ is a bounded solution to the linear equation
$\tilde{\mathcal{H}}u_{x_k}$ in $\Omega\cap C(X_0,t_0,r_0)$. Hence, by a Fatou type theorem, see for example Theorem 4.5 in \cite{N}, we deduce that
\begin{align}\label{conc2++}
\lim_{\substack{(Y,s)\in \Gamma(X,t)\\ (Y,s)\to (X,t)}} u_{x_k}( Y,s )
 \end{align}
 exists for $\omega$-a.e. $(X,t)\in\Sigma\cap C(X_0,t_0,r_0)$, where $\omega$ is the parabolic measure associated to $\tilde{\mathcal{H}}$. However,
 using that $\delta\lesssim u$ close to the boundary, \ Lemma \ref{Ainfty}, Theorem \ref{thm1-a} and Corollary \ref{thm1-b}, we see that the same conclusion must hold for  $\sigma$-a.e. $(X,t)\in\Sigma\cap C(X_0,t_0,r_0)$. This proves \eqref{conc2}. Next,  let $u\equiv 0$ on $(\mathbb R^{n+1}\setminus\Omega)\cap C(X_0,t_0,r_0)$, and consider the functional
 $$T(\phi):= -\iiint |\nabla_Xu|^{p-2}\nabla_Xu\cdot\nabla_X\phi -u \partial_t\phi\, \d X \d t,$$
 for $ \phi \in C_0^\infty(C(X_0,t_0,r_0))$. Using our assumptions on $u$, it follows readily that $T$ is a non-negative distribution, and
 hence there exists a locally finite measure $\mu$ supported on $\Sigma\cap C(X_0,t_0,r_0)$ such that
 $$T(\phi)=\iint \phi \d \mu,$$
 for $ \phi \in C_0^\infty(C(X_0,t_0,r_0))$. This proves \eqref{conc3}. Furthermore, given $\varepsilon>0$ small it follows from \eqref{boundsfree} and \eqref{boundsfreere} that
 $$\{(X,t): u(X,t)=\varepsilon\}\cap\Omega\cap C(X_0,t_0,r_0)$$
  is a smooth hypersurface and in particular that
 $$\{(X,t): u(X,t)=\varepsilon\}\cap\Omega\cap C(X_0,t_0,r_0)=\{(x,x_n,t):\ x_n=\psi_{\varepsilon}(x,t)\}\cap C(X_0,t_0,r_0),$$
 for a  Lip(1,1/2) function $\psi_\varepsilon$ with constant $b_1$ independent of $\varepsilon$. As a consequence
 $$\iiint_{\Omega_\varepsilon} |\nabla_Xu|^{p-2}\nabla_Xu\cdot\nabla_X\phi -u \partial_t\phi\, \d X \d t=-\iint_{\partial\Omega_\varepsilon} \phi \d \mu_\varepsilon, $$
 $\Omega_\varepsilon:=\{(x,x_n,t):\ x_n>\psi_{\varepsilon}(x,t)\}$, for all $ \phi \in C_0^\infty(C(X_0,t_0,r_0))$, and $ \d \mu_\varepsilon=|\nabla_Xu(X,t)|^{p-1}\d\sigma_\varepsilon(X,t)$ where
 $\sigma_\varepsilon$ is the surface measure on $\{(x,x_n,t):\ x_n=\psi_{\varepsilon}(x,t)\}$. Using \eqref{conc2} and dominated convergence we deduce that
 $$ \d \mu=|\nabla_Xu(X,t)|^{p-1}\d\sigma(X,t)\mbox{ on }\partial\Omega\cap C(X_0,t_0,r_0).$$
 This proves the second part of \eqref{conc4}.  It remains to prove the first part of \eqref{conc4}, i.e., that
 \begin{align*}
\nabla_Xu(X,t)=|\nabla_Xu(X,t)|\mathbf{n}(X,t),
\end{align*}
for  $\sigma$-a.e. $(X,t)\in\Sigma\cap C(X_0,t_0,r_0)$, and where $\mathbf{n}(X,t)$ was introduced in Definition \ref{ddeffa}. Consider $\Delta:=\Delta(X_0,t_0,r_0)$, $(X_0,t_0)\in\partial\Omega$. Let  $ E $ be the set of all $ (Z,\tau) \in \Delta $ which satisfies the following.
  \begin{align}\label{3.14}
    (a)& \hs{.2in}\mbox{$(Z,\tau) $  is a point of  density for $ E $ relative to  $ \si$. }\notag \\
   (b)& \hs{.2in} \mbox{There is a time-independent tangent plane $T(Z,\tau)$ to $\partial \Omega $ at  $(Z,\tau)$}\notag\\
   & \hs{.2in} \mbox{with unit normal $\mathbf{n}(Z,\tau)$.} \notag \\
   (c)& \hs{.2in} \mbox{$\lim_{\rho \rar 0} \rho^{-(n+1)} \si ( \Delta\cap C( Z,\tau, \rho) )= 2\hat a$}.\notag\\
  (d)& \hs{.2in}  \mbox{$\lim_{\rho \rar 0} \rho^{-(n+1)} \mu ( \Delta\cap C( Z,\tau, \rho) )= 2\hat a|\nabla_Xu(Z,\tau)|^{p-1}$.}
  \end{align}
  In  \eqref{3.14} $(c)$, $ \hat a, $  denotes the Lebesgue $(n-1)$-measure of the unit ball in $ \mathbb R^{n-1}. $  Now $ \si (\Delta \sem E ) = 0$.
    Indeed   $(a)$ of \eqref{3.14}  for $ \si$ almost every $(Z,\tau)$ is a consequence of the fact that  $ \si$ is a regular Borel measures and differentiation theory. To   prove  $(b)$ of \eqref{3.14} we  need to prove, for $ \si$ almost every $ (Z,\tau)\in\De $, that there exists a time-independent hyperplane
    $ T (Z,\tau)$, with unit normal $\mathbf{n}(Z,\tau)$, such that
            \begin{align} \label{3.16}  \lim_{\rho \rar 0} \frac{ H ( T ( Z,\tau ) \cap C ( Z,\tau, \rho ), \Delta(Z,\tau,\rho))}{ \rho} =  0,
            \end{align}
      where $H$  denotes parabolic Hausdorff distance. However, this follows from the Rademacher theorem stated in Theorem \ref{rademacher}, see \eqref{3.16tang} in Remark \ref{remtang}. $(c)$ of \eqref{3.14} is a consequence of the same argument, and $(d)$ was proved above.

    We now use a blow-up argument to complete the proof of  \eqref{conc4}. Let $ E$ and $\Delta$ be as in  \eqref{3.14} and
     consider $ (Z,\tau) \in E. $ Using invariance of the $p$-parabolic equation under spatial rotations, and invariance under translations in $(X,t)$, we may assume
    $ (Z,\tau) = 0 $,  $T ( 0,0 ) = \{ (X,t) \in \mathbb R^n :  x_n = 0\}$,  where $ T ( 0,0 ) $ is the
    time-independent tangent plane in  \eqref{3.14}, and that $\mathbf{n}(0,0)=e_n$.  We let $ H :=  \{ (X,t) \in \mathbb R^n :  x_n > 0\} $, $ \partial H= T ( 0,0 )$.  Let $ \{\rho_m\} $ be a decreasing sequence of positive numbers
    with limit zero and $\rho_1\ll r$.  Let
    \begin{align}\label{3.17}
    \Om_m  = \{ (X,t): (\rho_mX,\rho_m^pt) \in \Om \cap C ( 0,0, r  ) \},
    \end{align}
    and let
    \begin{align}
v_m ( X,t ) =  \rho_m^{-1}  \, u(\rho_mX,\rho_m^pt), \, (X,t)  \in  C ( 0,0, r  ).
 \end{align}
Fix  $ R >> 1. $ Then  for  $ m $ sufficiently large, say $ m \geq m_0, m_0 = m_0 (R), $  we note that
  $ v_m $ is a  $p$-parabolic function  in  $ \Om_m \cap C ( 0,0, R )$, continuous in $ C ( 0,0, R )$,  and $ v_m \equiv 0 $ on $ C ( 0,0, R ) \setminus \Om_m.$ Define
  \beq \label{3.18}  \nu_m ( G  ) =  \rho_m^{1-n}   \mu ( \rho_m G ),  \mbox{ whenever $ G $ is a Borel subset of $ C ( 0,0, R ) .$} \eeq Then $ \nu_m $ is the measure corresponding to  $ v_m $ as in \eqref{conc3} for
  $ m \geq m_0. $  Furthermore,
  \beq \label{3.19}   |\nabla_X v_m | \leq c  \mbox{ on } \Om_m\cap C(0,0,R), \eeq
  and
\beq \label{3.20}   | v_m ( X,t ) | \leq c\dist(X,t,\partial\Omega_m), \, \, (X,t) \in \Om_m \cap C(0,0,R).\eeq
 \eqref{3.16} implies that
     \beq \label{3.21}   H (  \Om_m \cap C ( 0,0, R ),    H \cap  C ( 0,0, R ) )  \rar 0  \mbox{ as }
      m \rar \infty.
    \eeq  From   \eqref{3.19}-\eqref{3.21}
    we see that a subsequence of $ \{v_m\}, $   denoted $ \{v'_m\} $ converges uniformly on
     compact subsets of  $ \mathbb R^{n+1} $ to a  H\"{o}lder continuous function $ v $ with $ v \equiv 0 $  in $ \mathbb R^{n+1} \sem H. $  Also $ v \geq 0 $ is  a $p$-parabolic function in $ H. $  Next, using Schwarz reflection, Theorem \ref{osci}, \eqref{3.19}, and \eqref{3.20}, we see that for each $R\gg 1$, we have
     $$  \osc_{ Q_{\varrho}^{c\varrho}} \nabla_Xv \leq 4c \left(\frac{\varrho}{R} \right)^{\alpha},
$$
for all $\varrho \in (0,R)$ with $\alpha$ as in Theorem~\ref{osci}. In particular, letting $R\to\infty$, we can conclude that $\nabla_Xv$ is constant on $H$. Hence
$v=\beta x_n^+$ for some constant $\beta>0$, $x_n^+=\max\{x_n,0\}$. Next, let
$\{\nu'_m\}$, be measures corresponding to $ \{v'_m\}$  in the sense that
$$\iiint |\nabla_Xv'_m|^{p-2}\nabla_Xv_m'\cdot\nabla_X\phi -v_m' \partial_t\phi\, \d X \d t=-\iint \phi \d \nu_m',$$
for all $\phi\in C_0^\infty(C(0,0,R))$ provided $ m \geq m_0, m_0 = m_0 (R)$. Then  the  measures, $\{\nu'_m\}$ have  uniformly bounded total masses  on $ C ( 0, 0, R ).$ Using this and \eqref{3.19}, \eqref{3.20}, we obtain  that    $ \nu'_m $ converges weakly to  $ \nu $ where $ \nu $ is the measure associated with
     $ \al \, x_n^+ . $  It follows
        $ \nu  = \al^{p-1}  \si_H $ where $ \si_H, $ denotes the $(n + 1)$-dimensional parabolic Hausdorff measure on $ H. $ Using this computation, weak convergence, \eqref{3.18}, and  \eqref{3.14} $(d)$, we get
     \[  2\al^{p-1} \, \hat a \, R^{n+1}  =   \lim_{m \rar \infty}
       \nu'_m ( C (0,0,R)) =  \lim_{m \rar \infty} s_m^{-1-n} \mu ( C ( 0, 0, R s_m ) ) =
         2\hat a \, R^{n+1} |\nabla_Xu|^{p-1} ( 0,0 ). \]
         Hence,
         \begin{align}\label{3.22} \al= |\nabla_Xu|( 0,0 ).
         \end{align}
         From \eqref{3.22} and our earlier observations we see that  $ (X,t) \rightarrow \rho^{-1} v ( \rho X,\rho^pt ) $  converges uniformly as $ \rho \rightarrow 0 $ to $ \al x_n^+$  on  compact subsets of $ \mathbb R^{n+1} $  and $ (X,t) \rightarrow  \nabla_Xv( \rho X,\rho^pt )  $ converges uniformly to $ \al  e_n $  as $ t \rightarrow 0 $, when $ (X,t) $  lies in a compact subset of  $ H. $  Put together these observations prove  \eqref{conc4}.\\

\end{document}